\documentclass[notitlepage,twoside,a4paper]{amsart}
\usepackage{mathrsfs}
\usepackage{amsmath,amssymb,enumerate}
\usepackage{epsfig,fancyhdr,color}
\usepackage{epstopdf}
\usepackage{amssymb}
\usepackage{amsmath,amsthm}
\usepackage{latexsym}
\usepackage{amscd}
\usepackage{psfrag}
\usepackage{graphicx}
\usepackage{epsf}
\usepackage[latin1]{inputenc}
\usepackage[all]{xy}
\usepackage{tikz}
\usepackage{prettyref}
\usepackage{subfigure}
\usepackage{float}
\usepackage{color}
\usepackage{array}
\usepackage{cite}
\usepackage[totalwidth=16cm,totalheight=24cm,marginparwidth=40pt]{geometry}
\usepackage{hyperref}
\usepackage{MnSymbol}

\newcommand{\II}{I\hspace{-0.1cm}I}

\newcommand{\Pidang}{\cP^{\rm ideal,ang}_{\Gamma,\theta}}
\newcommand{\Pidlen}{\cP^{\rm ideal,len}_{\Gamma,l}}
\newcommand{\Pidmet}{\cP^{\rm ideal,met}_{h}}
\newcommand{\Pidn}{\cP^{\rm ideal}_{n}}

\newcommand{\Pcomet}{\cP^{\rm cpt,met}_{h}}
\newcommand{\Pcon}{\cP^{\rm cpt}_{n}}
\newcommand{\Pcodua}{\cP^{\rm cpt,dual}_{h^*}}
\newcommand{\Phyang}{\cP^{\rm hyper,ang}_{\Gamma,\theta}}

\newcommand{\Phymet}{\cP^{\rm hyper,met}_{h}}
\newcommand{\VRcomet}{V_R^{\rm cpt,met}}
\newcommand{\VRcodua}{V_R^{\rm cpt,dual}}
\newcommand{\VRidang}{V_R^{\rm ideal,ang}}

\newtheorem{theorem}{\rm\bf Theorem}[section]

\newtheorem{conjecture}[theorem]{\rm\bf Conjecture}
\newtheorem{lemma}[theorem]{\rm\bf Lemma}

\newtheorem{definition}[theorem]{\rm\bf Definition}

\newtheoremstyle{named}{}{}{\itshape}{}{\bfseries}{.}{.5em}{#1 \thmnote{#3}}
\theoremstyle{named}

\newcommand{\C}{{\mathbb C}}
\newcommand{\CP}{{\mathbb CP}}

\newcommand{\HH}{{\mathbb H}}
\newcommand{\R}{{\mathbb R}}

\newcommand{\Z}{{\mathbb Z}}

\newcommand{\cC}{{\mathcal C}}

\newcommand{\bE}{{\mathbb E}}
\newcommand{\cE}{{\mathcal E}}

\newcommand{\cH}{{\mathcal H}}

\newcommand{\cS}{{\mathcal S}}

\newcommand{\cP}{{\mathcal P}}
\newcommand{\cT}{{\mathcal T}}

\newcommand{\PSL}{\rm{PSL}}

\newcommand{\tr}{\rm{tr}}
\newcommand{\dev}{\rm{dev}}
\newcommand{\Sigeps}[1]{{\Sigma_{Eps,#1}}}

\newcounter{notes}%

\def\interieur#1{\mathord{\mathop{\kern 0pt #1}\limits^\circ}}

\title[]{Projective rigidity  of circle patterns and polyhedral surfaces in hyperbolic ends}

\author{Jean-Marc Schlenker}
\address{Jean-Marc Schlenker:
University of Luxembourg, FSTM, Department of Mathematics, 
Maison du nombre, 6 avenue de la Fonte,
L-4364 Esch-sur-Alzette, Luxembourg}
\email{jean-marc.schlenker@uni.lu}

\date{v0, \today}

\begin{document}

\begin{abstract}
  Let $S$ be a closed, orientable surface of genus $g\geq 2$. We consider Delaunay circle patterns on $S$ equipped with a complex projective structure. We prove that the space of complex projective structures on $S$ equipped with a Delaunay circle pattern of prescribed combinatorics and intersection angles is a manifold of dimension $6g-6$, and that the forgetful map to the space $\cC_S$ of $\CP^1$-structures on $S$ is a Lagrangian immersion. This extends a recent result of Bonsante and Wolf for circle packings.

  This statement, and its proof, are more conveniently stated in terms of ideal polyhedral surfaces (surfaces with vertices at infinity) in hyperbolic ends, with the angles between the circles corresponding to the dihedral angles. Seen from this angle, we extend the statement to ideal polyhedral surfaces with prescribed edge lengths (or induced metrics), and to other types of polyhedral surfaces, either compact or hyperideal. 
\end{abstract}

\maketitle

\tableofcontents
 
\section{Main statements}

\subsection{Complex projective structures on surfaces and the KMT conjecture}
\label{ssc:kmt}

We consider here a closed, oriented surface $S$ of genus at least $2$. We are interested in the circle packings, or more generally circle patterns, which can be considered on $S$ equipped with a complex projective structure (or $CP^1$-structure), see Section \ref{ssc:cp1} for the definitions. We denote by $\cC_S$ the space of $CP^1$-structures on $S$, considered up to isotopy, and by $\cT_S$ the Teichm\"uller space of $S$, defined as the space of complex structures on $S$, considered also up to isotopy. Any $CP^1$-structure on $S$ has an underlying complex structures, so there is well-defined ``forgetful'' map $\pi:\cC_S\to \cT_S$.

The space $\cC_S$ is moreover naturally equipped with a (complex) symplectic structure. There are in fact several possible symplectic structures that can be considered on $\cC_S$, but we will be particularly interested here in one such symplectic form, denoted by $\omega$, which is defined through an identification of $\cC_S$ with $T^*\cT_S$, see Section \ref{ssc:symplectic}. 

There is a natural notion of disk on a surface $S$ equipped with a $CP^1$-structure $\sigma$, and therefore a natural notion of circle packing, which can be defined as a collection of disjoint open disks. Given a circle packing, its {\em nerve} is defined as the graph embedded in $S$ that has a vertex for each disk, and an edge between two vertices if and only if the corresponding disks are tangent. 

Let $\Gamma$ be a graph embedded in $S$, which is the 1-skeleton of a triangulation. We will denote by $\cC_\Gamma$ the space of pairs $(\sigma, C)$, where $\sigma$ is a $CP^1$-structure on $S$ equipped with a circle packing $C$ with nerve $\Gamma$. We denote by $\pi_\Gamma=\pi_{|\cC_\Gamma}:\cC_\Gamma\to \cT_S$ the restriction of $\pi$ to $\cC_\Gamma$.

We are motivated by the following remarkable statement, proposed by Kojima, Mizushima and Tan \cite{KMT,KMT2,KMT3}.

\begin{conjecture}[KMT conjecture] \label{cj:kmt}
  Let $\Gamma$ be the 1-skeleton of a triangulation of $S$. Then $\pi_\Gamma:\cC_\Gamma\to \cT_S$ is a homeomorphism.
\end{conjecture}

This statement is true when $S$ has genus $0$, it is then equivalent to the classical Koebe circle packing theorem \cite{koebe}. For surfaces of genus $1$, it was recently proved by Wayne Lam \cite{lam2019quadratic,lam2024space}. For surfaces of higher genus, it was recently proved by Bonsante and Wolf \cite{bonsante-wolf:conformal}, under a somewhat restrictive combinatorial condition on $\Gamma$ (all vertices should have degree congruent to $2$ modulo $4$). 

We are in fact motivated here by an extension of Conjecture \ref{cj:kmt}, from circle packings to Delaunay circle patterns, in the following sense.

\begin{definition}[Delaunay circle patterns]
  A {\em Delaunay circle pattern} on $(S, \sigma)$ is a family of open disks $D=(D_1,\cdots, D_v)$ embedded in $(S,\sigma)$ such that:
  \begin{itemize}
  \item $S\setminus(\cup_{i=1}^v D_i)=F$ is a finite set,
  \item each of the lift to $\tilde S$ of one of the $D_i, 1\leq i\leq v$ has at least 3 points of the lift of $F$ in its boundary.
  \end{itemize}
  We will always consider Delaunay circle patterns up to relabeling their disks.
\end{definition}

\begin{definition}[Intersection data of a Delaunay circle pattern]
  Given a Delaunay circle pattern $D=(D_1, \cdots, D_v)$, its {\em intersection data} is the pair $\Gamma,\theta)$ defined as follows.
  \begin{itemize}
  \item $\Gamma$ has
    \begin{itemize}
    \item one vertex for each of the $D_i$,
    \item an edge between two vertices corresponding to $D_i$ and $D_j$ ($i\neq j$) if and only if $D_i\cap D_j\neq \emptyset$ and the two points of $\partial D_i\cap \partial D_j$ are in $F$.
    \end{itemize}
  \item $\theta:E(\Gamma)\to (0,\pi)$ associates to an edge $e$ between the vertices corresponding to $D_i$ and $D_j$ the angle of the complement of $D_j$ in $D_i$ (or equivalently the complement of $D_i$ in $D_j$).
  \end{itemize}
\end{definition}

The intersection data of a Delaunay circle pattern is an admissible weighted graph, in the following sense -- see Section \ref{ssc:delaunay} for a proof of this rather classical fact.

\begin{definition}[Admissible weighted graphs] \label{df:weighted}
  An {\em admissible weighted graph} in $S$ is a graph $\Gamma$ embedded in $S$ which is the 1-skeleton of a cell decomposition, equipped with a function $\theta:E(\Gamma)\to (0,\pi)$ satisfying the conditions that
  \begin{itemize}
  \item the sum of the values of $\theta$ on the boundary of any face is $2\pi$,
  \item for each contractible cycle $\gamma$ in $\Gamma$ which is not contained in the boundary of a face, the sum of the values of $\theta$ is strictly larger than $2\pi$. 
  \end{itemize}
  Given a weighted graph $(\Gamma,\theta)$, we denote by $\cC_{\Gamma,\theta}$ the space of pairs $(\sigma, \cC)$ where $\sigma$ is a complex projective structure on $S$ and $\cC$ is a circle pattern on $(S,\sigma)$ with nerve $\Gamma$ and intersection angles given by $\theta$.
\end{definition}

The KMT conjecture, in the generalized form from \cite{delaunay}, is the following.

\begin{conjecture}[generalized KMT] \label{cj:kmtg}
  Let $(\Gamma,\theta)$ be an admissible weighted graph in $S$. Then the forgetful map from $\cC_{\Gamma,\theta}\to \cT_S$ is a homeomorphism.
\end{conjecture}

Conjecture \ref{cj:kmtg} has been proved, for tori, by Wayne Lam \cite{lam2019quadratic,lam2024space}.

\subsection{Complex projective structures with circle packing of given nerve}
\label{ssc:circle-packings}

An important motivation here is a recent result of Bonsante and Wolf \cite{bonsante-wolf:projective} who proved that the space of $\CP^1$-structures on a surface, admitting a circle packing with given combinatorics, is a submanifold of the space of the space $\cC_S$ of complex projective structures on $S$·

\begin{theorem}[Bonsante--Wolf] \label{tm:bonsante-wolf}
  Let $T$ be a quasi-simplicial triangulation on $S$, a surface of genus $g(S) \geq 2$. Then
  \begin{enumerate}
  \item The moduli space $\cC_T$ of pairs of projective structures and circle packings with nerve $T$ admits a natural manifold structure of of dimension $6g(S)-6$, so that
  \item The projection of $\cC_T$ to the space (of holonomies) of projective structures is a smooth immersion.
  \end{enumerate}
\end{theorem}

This statement can be considered as a step towards a possible proof of Conjecture \ref{cj:kmt} -- a key step would in fact to prove a similar rigidity statement, but when one only fixes the complex structure of the $\CP^1$-structure, rather than the whole $\CP^1$-structure.

\subsection{Spaces of of Delaunay circle patterns}
\label{ssc:patterns}

We provide an extension of Theorem \ref{tm:bonsante-wolf} to Delaunay circle patterns, and add an additional Lagrangian property.

\begin{theorem} \label{tm:prorig-delaunay}
  Let $(\Gamma,\theta)$ be an admissible weighted graph in $S$. Then $\cC_{\Gamma,\theta}$ is an analytic manifold of dimension $6g-6$. Moreover, the restriction $\rho\circ i:\cC_{\Gamma,\theta}\to \cC_S$ is a smooth immersion which is Lagrangian for the symplectic form $\omega$ on $\cC_S$.
\end{theorem}

This statements has three parts. The first is the ``projective rigidity'' of Delaunay circle patterns, that is, the fact that they cannot be deformed with fixed intersection angles in a surface equipped with a fixed complex projective structure. The second part is that when allowing for variation of the complex projective structure, the deformation space has the expected dimension, that is, $6g-6$. It then follows that it is an (immersed) manifold. The third part, which is proved using different tools, is the fact that the image is Lagrangian in $(\cC_S,\omega)$. This last point can be compared to recent results of Lam \cite{lam:pullback}, for which other symplectic forms are considered.

Theorem \ref{tm:prorig-delaunay} is an extension of Theorem \ref{tm:bonsante-wolf} because circle packings can be considered as special cases of Delaunay circle patterns, see Section \ref{ssc:special}. However, the proof presented here is perhaps simpler than that of \cite{bonsante-wolf:projective}.

Theorem \ref{tm:prorig-delaunay} can be extended to more general types of circle patterns: either hyperideal circle patterns as in \cite{hcp,springborn-hcp,bobenko-lutz}, or patterns of circles allowing for more intersections (see \cite{bowers:combinatorics}). This follows from Theorems \ref{tm:hyperideal*} and \ref{tm:compact*} below on polyhedral surfaces. We will however focus on the polyhedral surfaces point of view, rather than on the (equivalent) description in terms of circle patterns, which also leads to other, dual statements (Theorem \ref{tm:ideal*}, \ref{tm:compact} and \ref{tm:hyperideal}), in terms of edges lengths rather than angles.

\subsection{Polyhedral surfaces in hyperbolic ends}

Theorem \ref{tm:prorig-delaunay} can be stated in terms of polyhedral surfaces in hyperbolic ends, and it is in this setting that the proofs will be easier to understand. We start by some definitions.

\begin{definition}[Hyperbolic ends]
  A {\em hyperbolic end} is a non-complete 3-dimensional hyperbolic manifold, homeomorphic to $S\times [0,\infty)$ for $S$ a closed surface of genus at least $2$, complete on the side corresponding to $\infty$, and bounded by a concave pleated surface on the side corresponding to $0$. 
  We denote by $\cE_S$ the space of hyperbolic metrics $g$ on $S\times [0,\infty)$ such that $(S\times [0,\infty), g)$ is a hyperbolic end, considered up to isotopy.
\end{definition}

Let $M$ be a quasifuchsian manifold (or more generally a convex co-compact hyperbolic manifolds), and let $C(M)$ be its convex core. Then each connected component of $M\setminus \rm{Int}(C(M))$ is a hyperbolic end. However some hyperbolic ends are not obtained in this manner, and are not isometric to a subset of a quasifuchsian manifold.

Given a hyperbolic end $\cE$, isometric to $(S\times [0,\infty), g)$ for a hyperbolic metric $g$, we will denote by $\partial_0\bE$ the boundary component of $\bE$ corresponding to $S\times \{ 0\}$, and by $\partial_\infty \bE$ the ideal boundary of $\bE$, corresponding to $S\times \{ \infty\}$. This ideal boundary is equipped with a $\CP^1$-structure, and there is a one-to-one correspondence between $\CP^1$-structures on $S$ and hyperbolic ends homeomorphic to $S\times [0,\infty)$, see Section \ref{ssc:complex-ends}.

We will use the following (standard) definition. A subset $K\subset \bE$ is {\em geodesically convex} if any geodesic segment $\gamma\subset \bE$ with endpoints in $K$ is contained in $K$. A convex subset is {\em proper} if it contains a neighborhood of $\partial_0\bE$. Unless otherwise specified, all the convex subsets we will consider will be closed in $\bE$. If $F\subset \bE$, the {\em convex hull} $CH(F)$ of $F$ in $\bE$ is the smallest convex subset of $\bE$ containing $F$. If $F\subset \bE\cap \partial_\infty \bE$, $CH(F)$ is the smallest closed convex subset $K$ of $\bE$ such that $F\subset K\cup \partial_\infty K$.

\begin{definition}[Ideal polyhedral surface]
   Let $E\in \cE_S$ be a hyperbolic end. An {\em ideal polyhedral surface} in $\cE$ is the boundary of a proper convex subset which is the convex hull of a finite set of points in $\partial_\infty \bE$.
\end{definition}

It follows from the definition (and in particular of the properness hypothesis) that ideal polyhedral surfaces are locally like an ideal hyperbolic polyhedron. Each point has a neighborhood which is contained in either in a totally geodesic plane, or in the union of two totally geodesic half-planes intersecting along their common boundary. Moreover, those ideal polyhedral surfaces are the union of {\em faces} which are isometric to ideal polygons, with disjoint interiors.

We are particularly interested in the sets of ideal polyhedral surfaces with given combinatorics and dihedral angles. Unless otherwise specified, we will always consider the {\em exterior} dihedral angles, that is, the angle is $0$ if two faces are coplanar with disjoint interiors.

\begin{definition}
  Let $(\Gamma,\theta)$ be an admissible weighted graph. We denote by $\Pidang$ the space of pairs $(\bE,\Sigma)$, where $\bE$ is a hyperbolic end and $\Sigma$ is an ideal polyhedral surface in $\bE$ with 1-skeleton  $\Gamma$ and exterior dihedral angles given by $\theta$.
\end{definition}

There is a one-to-one relation between Delaunay circle patterns and ideal polyhedral surfaces in hyperbolic ends, see Section \ref{ssc:ideal}. Moreover, intersection angles between the circles corresponds to the exterior dihedral angles of the corresponding ideal polyhedral surface.


Theorem \ref{tm:prorig-delaunay} is therefore equivalent to the following statement in terms of ideal polyhedral surfaces in hyperbolic ends.

\begin{theorem} \label{tm:ideal}
  Let $(\Gamma,\theta)$ be an admissible weighted graph in $S$. Then $\Pidang$ admits a natural analytic manifold structure of dimension $6g-6$, and its projection to $\cE_S$ is a smooth immersion which is Lagrangian for $\omega$.
\end{theorem}

The projection from $\Pidang$ to $\cT_S$, sending a hyperbolic end to its conformal metric at infinity, is known to be proper \cite{delaunay}. The key point missing to prove Conjecture \ref{cj:kmtg} is therefore that the image of $\Pidang$ in $\cE_S$ is never tangent to the fibers of the projection from $\cE_S$ to $\cT_S$


\subsection{Edge lengths and induced metrics of ideal polyhedral surfaces}

There is another possible data one can consider on an ideal polyhedral surface $\Sigma$ in a hyperbolic end $\bE$, namely, the ``lengths'' of its edges. Since the vertices are at infinity, all edges have infinite length, so one needs to ``truncate'' a neighborhood of each vertex by removing a horoball centered at that vertex, and the length of an edge is then defined as the (oriented) distance between the horoballs associated to its endpoints. Replacing the horoball by another horoball centered at the same vertex $v$ adds (or substracts) a constant to the lengths of all edges adjacent to $v$, so the edge lengths of an ideal polyhedral surface $\Sigma\subset \bE$ is defined as an element of $\R^E/i(\R^V)$, where $E$ is the set of edges of $\Sigma$, $V$ is the set of its vertices, and $i:\R^V\to \R^E$ is the linear map sending $(a_v)_{v\in V}\in \R^V$ to $(a_{e_-}+a_{e_+})_{e\in E}\in \R^E$, where $e_-$ and $e_+$ denote the end points of an edge $e$.

\begin{definition}
  Let $l\in\R^E/i(\R^V)$. We denote by $\Pidlen$ the space of pairs $(\bE,\Sigma)$, where $\bE$ is a hyperbolic end and $\Sigma$ is an ideal polyhedral surface in $E$ with combinatorics $\Gamma$ and edge lengths given by $l$.
\end{definition}

The following statement can be considered as dual, in a precise way which will appear in Section \ref{ssc:ideal-dual}, to Theorem \ref{tm:ideal}. It could also be stated in terms of Delaunay circle patterns on surfaces equipped with a $\CP^1$-structure.

\begin{theorem} \label{tm:ideal*}
  Let $\Gamma$ be the 1-skeleton of a triangulation of $S$, and let $l\in \R^E/i(\R^V)$. Then $\Pidlen$ admits a natural analytic manifold structure of dimension $6g-6$, and its projection to $\cE_S$ is a smooth immersion. 
\end{theorem}

This statement applies a bit beyond convex ideal polyhedral surfaces. It still applies to a triangulated ideal polyhedral surface $\Sigma$ for which some edges of $\Gamma$ are ``strictly convex'', while others are diagonals of faces, so that the corresponding exterior dihedral angle is $0$. Those polyhedral surfaces are boundary points of the space of convex ideal polyhedral surfaces with 1-skeleton $\Gamma$. However they are interior points of a larger manifold of ideal polyhedral surfaces with 1-skeleton $\Gamma$ which are not necessarily convex, that is, for which the exterior dihedral angle at some of the edges might be negative. 

Theorem \ref{tm:ideal} and Theorem \ref{tm:ideal*} are quite closely related, and there is an argument, based on the complex structure on the space of ideal polyhedral surfaces in a hyperbolic end, that shows the equivalence between the two, see Section \ref{ssc:complex-ideal}.

However, Theorem \ref{tm:ideal*} is essentially local in nature, since deformations with prescribed edge lengths can lead to non-convex polyhedra, for which Theorem \ref{tm:ideal*} does not apply. Theorem \ref{tm:ideal-m} below, although closely related, has more global implications.

When $C\subset S$ is a finite subset, we denote by $\cT_{S,C}$ the Teichm\"uller space of hyperbolic metrics on $S$ with cusps at the points of a finite subset $C$. If $\Sigma\subset \bE$ is an ideal polyhedral surface in a hyperbolic end with vertices at the points of a finite subset $C\subset S$, then the induced metric on $\Sigma$ is a hyperbolic metric with cusps at the points of $C$. This motivates the following definition.

\begin{definition}
  Let $C\subset S$ be a set of $n$ points, and let $h$ be a complete hyperbolic metric on $S$ with cusps at the points of $C$. We denote by $\Pidmet$ the space of pairs $(\bE,\Sigma)$, where $\bE\in \cE_S$ is a hyperbolic end and $\Sigma$ is an ideal polyhedral surfaces in $\bE$ with induced metric isometric to $h$. 
\end{definition}

\begin{theorem} \label{tm:ideal-m}
  Let $C\subset S$ be a subset of cardinal $n\geq 1$, and let $h\in \cT_{S,C}$.  Then $\Pidmet$ admits a natural $C^{1,1}$ manifold structure of dimension $6g-6$, and its projection to $\cE_S$ is a smooth immersion. 
\end{theorem}

The lower regularity in this statement is a consequence of the fact that the induced metric on an ideal polyhedral surface is only $C^{1,1}$ as a function of the position of the vertices, it is smooth as long as the polyhedral surface is triangulated, but has lower regularity when the combinatorics changes.





\subsection{Compact and hyperideal polyhedral surfaces}

We also provide analogs of Theorems \ref{tm:ideal}, \ref{tm:ideal*} and \ref{tm:ideal-m} to two similar settings, one concerning ``compact'' polyhedral surfaces, the other ``hyperideal'' polyhedral surfaces. In both cases, the statements could be translated in terms of some kind of circle patterns, somewhat different from the Delaunay circle patterns appearing in Theorem \ref{tm:prorig-delaunay}: for hyperideal polyhedral surfaces, they are the ``hyperideal'' circle patterns appearing for instance in \cite{hcp,springborn-hcp,bobenko-lutz}, while for compact polyhedral surfaces, it would be the patterns of overlapping circles (see e.g. \cite{bowers:combinatorics} for a wide perspective on this topic). In both cases, we provide in fact two statements, one concerning the induced metrics on those surfaces, the other concerning the dihedral angles (for hyperideal polyhedral surfaces), or, for compact polyhedral surfaces, the ``dual metrics'', which is more directly related to the dihedral angles of ideal polyhedral surfaces.

\subsubsection{Compact polyhedral surfaces}

Compact polyhedral surfaces are defined as surfaces that look locally like the boundary of a compact polyhedron in $\HH^3$.

\begin{definition}[Compact polyhedral surfaces] \label{df:polycompact}
  Let $\bE\in \cE_S$ be a hyperbolic end. A {\em compact polyhedral surface} in $\bE$ is a surface $\Sigma\subset \bE$ which bounds a proper, compact, convex region including $\partial_0\bE$ and such that each $x\in \Sigma$ has a neighborhood $U\subset \bE$ such that $\Sigma\cap U$ corresponds to the intersection of the boundary of a compact polyhedron in $\HH^3$ with an open subset of $\HH^3$ isometric to $U$.
\end{definition}

Alternatively, a compact polyhedral surface can be defined as the convex hull in $\bE$ of a finite, non-empty set of points $V$, that is, the smallest geodesically convex subset of $E$ containing the points of $V$, under the condition that it is a proper convex subset. Some of those points will then be the vertices.

A compact polyhedral surface is equipped with its {\em induced metric}, which is a hyperbolic metric with cone singularities of angle less than $2\pi$ at the vertices. It also carries a ``dual metric'', which has the same relation to the exterior dihedral angles as the induced metric has to the edge lengths, see Section \ref{ssc:compact-dual}. The dual metric is a spherical metric with cone singularities of angle larger than $2\pi$, and with closed, contractible geodesics of length larger than $2\pi$, see Section \ref{ssc:compact-dual}.

\begin{definition}
  \begin{itemize}
  \item Let $h$ be a hyperbolic metric on $S$ with cone singularities of angle less than $2\pi$.  We denote by $\Pcomet$ the space of pairs $(\bE,\Sigma)$, where $\bE\in \cE_S$ is a hyperbolic end and $\Sigma\subset \bE$ is a compact polyhedral surface in $\bE$ with induced metric isotopic to $h$.
  \item Let $h^*$ be a spherical metric on $S$ with cone singularities of angle larger than $2\pi$, and closed, contractible geodesics of length larger than $2\pi$.  We denote by $\Pcodua$ the space of pairs $(\bE,\Sigma)$, where $\bE\in \cE_S$ is a hyperbolic end and $\Sigma\subset \bE$ is a compact polyhedral surface in $\bE$ with dual metric isotopic to $h^*$.
  \end{itemize}
\end{definition}

Both $\Pcomet$ and $\Pcodua$ are subset of the space of hyperbolic ends equipped with $n$ points in convex position (which correspond to the vertices). They are in this manner both equipped with a manifold structure. However, both are only $C^{1,1}$, because the function that associates to a set of points in convex position the induced metric on the boundary of their convex hull is only $C^{1,1}$, see Section \ref{ssc:compact-dual}.

\begin{theorem} \label{tm:compact}
  Let $h$ be a hyperbolic metric on $S$ with cone singularities of angle less than $2\pi$. Then $\Pcomet$ is a $C^{1,1}$ manifold of dimension $6g-6$, and the projection on the first factor from $\Pcomet$ to $\cE_S$ is an immersion which is Lagrangian for $\omega$.
\end{theorem}

\begin{theorem} \label{tm:compact*}
  Let $h^*$ be a spherical metric on $S$ with cone singularities of angle larger than $2\pi$, and closed, contractible geodesics of length larger than $2\pi$. Then $\Pcodua$ is a $C^{1,1}$  manifold of dimension $6g-6$, and the projection on the first factor from $\Pcodua$ to $\cE_S$ is an immersion which is Lagrangian for $\omega$.
\end{theorem}

\subsubsection{Hyperideal polyhedral surfaces}

Similar results can be stated for hyperideal polyhedral surfaces. The definitions are in Section \ref{ssc:hyperideal}. The most natural (but perhaps not simplest) definition is based on the fact that a hyperbolic end $\bE$ can be naturally embedded in a larger {\em projective} manifold $\bar \bE$, which can be decomposed as
$$ \bar \bE = \bE\cup \partial_\infty \bE \cup \bE^*~, $$
where $\bE^*$ is equipped with a de Sitter metric which makes it into a globally hyperbolic maximal compact (GHMC) de Sitter space with the same holonomy and ideal boundary as $\bE$, see Section \ref{ssc:ends}. Hyperideal polyhedral surfaces can then be considered in $\bar \bE$ (equipped with its projective structure). 


\begin{definition} \label{df:polyhyperideal}
  Let $\bE\in \cE_S$ be a hyperbolic end. A {\em hyperideal polyhedral surface} in $\bE$ is the intersection with $\bE$ of a closed polyhedral surface in $\bar \bE$ with all vertices in $\bE^*$, and all edges intersecting $\bE$.  
\end{definition}

The {\em induced metric} on a hyperideal  polyhedral surface $\Sigma\in \bE$ is a complete metric on $S$ minus $|V|$ disks, where $V$ is the set of vertices of $\Sigma$, with ends of infinite area corresponding to each of the vertices. Here by {\em polyhedral surface} in $\bar\bE$ we mean again the boundary of a proper convex hull of $|V|$ points, now in $\bar\bE$. (Note that some definition allow for ideal ends in a hyperideal polyhedral surface. We will not do this here, so that the surfaces we consider might more adequately be called {\em strictly hyperideal} polyhedral surfaces.)

There is a notion of dual metric associated to a hyperideal polyhedral surface. It will however not be necessary here, since dual metrics of hyperideal polyhedral surfaces (or simply of hyperideal hyperbolic polyhedra) are uniquely determined by the dihedral angles, and conversely \cite{rousset1}. 

Hyperideal polyhedral surfaces have exterior dihedral angles which, considered as a function $\theta$ from the set of edges $E$ of $\Gamma$ to $(0,\pi)$, satisfy two conditions, see \cite{bao-bonahon}. (Note that the conditions here are slightly more restrictive than those of \cite{bao-bonahon} since we don't allow for ideal vertices.)
\begin{enumerate}
\item For every closed, contractible path $\gamma$ in $\Gamma^*$, $\sum_{e\in \gamma}\theta(e)>2\pi$,
\item for any path $\gamma$ in $\Gamma^*$ starting and ending at a face $f$ of $\Gamma^*$, and such that $\gamma$ is homotopic to a path in $f$, $\sum_{e\in \gamma}\theta(e)>\pi$.
\end{enumerate}
A graph $\Gamma$ equipped with a function $\theta:E\to (0,\pi)$ satisfying conditions (1) and (2) will be called an {\em admissible hyperideal weighted graph}. 

\begin{definition}
  \begin{itemize}
  \item Let $h$ be a complete hyperbolic metric on $S$ minus $|V|$ disks, with each end of infinite area.  We denote by $\Phymet$ the space of pairs $(\bE,\Sigma)$, where $\bE\in \cE_S$ is a hyperbolic end and $\Sigma\subset \bE$ is a hyperideal polyhedral surface in $\bE$ with induced metric isotopic to $h$.
  \item Let $(\Gamma, \theta)$ be an admissible hyperideal weighted graph in $S$. We denote by $\Phyang$ the space of pairs $(\bE,\Sigma)$, where $\bE\in \cE_S$ is a hyperbolic end and $\Sigma\subset \bE$ is a compact polyhedral surface in $\bE$ with combinatorics given by $\Gamma$ and exterior dihedral angles given by $\theta$.
  \end{itemize}
\end{definition}

\begin{theorem} \label{tm:hyperideal}
    Let $h$ be a complete hyperbolic metric on $S\setminus F$, with all ends of infinite area. Then $\Phymet$ is a $C^{1,1}$ manifold of dimension $6g-6$, and the projection on the first factor from $\Phymet$ to $\cE_S$ is an immersion which is Lagrangian for $\omega$.
\end{theorem}

\begin{theorem} \label{tm:hyperideal*}
  Let $(\Gamma, \theta)$ be an admissible hyperideal weighted graph in $S$. Then $\Phyang$ is a smooth manifold of dimension $6g-6$, and the projection on the first factor from $\Phyang$ to $\cE_S$ is an immersion which is Lagrangian for $\omega$.
\end{theorem}

\subsection{A wider perspective}
\label{ssc:perspective}

The questions considered here can be put in the framework of the Weyl problem, and the dual Weyl problem, for locally convex immersed surfaces in $\HH^3$. This point of view is \cite[Section 5]{weylsurvey}. We do not elaborate more on this point of view here.

\subsection{Acknowledgements}

I would like to thank Wayne (Wai Yeung) Lam for several interesting conversations related to the content of this paper.

\section{Arguments}

The arguments used here are significantly simpler when considering compact polyhedral surfaces, or possibly hyperideal polyhedral surfaces. We first describe the arguments for compact polyhedral surfaces, which is heuristically a bit simpler, then for hyperideal surfaces, and lastly for ideal surfaces. 

\subsection{Compact polyhedral surfaces}

To prove that the space of hyperbolic ends containing a compact polyhedral surface with prescribed induced metric is a $C^{1,1}$ Lagrangian submanifold of $\cE_S$ of dimension $6g-6$, (Theorem \ref{tm:compact}), we use the following steps.
\begin{enumerate}
\item The statement can be reduced to the corresponding statement for hyperbolic ends containing a compact, {\em triangulated} surface with prescribed edge lengths. Indeed when two or more of the faces become coplanar, the linear conditions satisfied by the end and position of the vertices are parallel, so that the corresponding submanifolds can be ``glued'' to obtain a $C^{1,1}$ submanifold.
\item We apply a rigidity argument invented by I. Pak \cite{pak:short} for Euclidean polyhedra to show that polyehdral surfaces cannot be deformed (preserving either the induced metric or the dual metric) in a fixed hyperbolic end. (The argument presented here is somewhat extended and adapted, to allow for polyhedra with faces of degree larger than $3$.)
\item We introduce a description of the infinitesimal deformations fixing the induced metric but varying the holonomomy (and therefore the hyperbolic end), in terms of variations of dihedral angles. We show this description is {\em dual} to that of deformations in a fixed hyperbolic end, in the sense that the two corresponding linear operators are {\em adjoint}. It therefore follows from point (2) that the deformation space when varying the holonomy has maximum rank, that is, $6g-6$. Therefore, the space of realizations (without fixing the holonomy) is a manifold of dimension $6g-6$.
\item It then follows from point (2) that the projection of this space to the space of hyperbolic ends is an immersion.
\item The proof that the submanifold obtained in this way is Lagrangian uses a different argument, based on a the first-order variation of a variant of the renormalized volume adapted to this setting. This variational formula can be interpreted as the fact that, on the submanifold of hyperbolic ends obtained, the differential of this ``volume'' is equal to the Liouville form $\lambda$ of $(\cE_S,\omega)$. It then follows that $\omega=d\lambda$ is zero on the submanifold. (This part is done in Section \ref{sc:lagrangian}.) 
\end{enumerate}

The proof of Theorem \ref{tm:compact*}, concerning the dual metrics of compact polyhedral surfaces, is quite similar. The main difference is that the proof takes place not on the compact polyhedral surface $\Sigma$, but on the dual surface $\Sigma^*$, which is contained in the de Sitter part $\bE^*$ of the hyperbolic end $\bar \bE$.

\subsection{Hyperideal hyperbolic surfaces}

The proof of Theorem \ref{tm:hyperideal} is again very similar, with the difference that the vertices of the hyperideal polyhedral surface $\Sigma$ are now contained in the de Sitter component $\bE^*$.

Theorem \ref{tm:hyperideal*} is again proved along the same lines, using the dual surface $\Sigma^*$, see Section \ref{ssc:hyperideal}.

\subsection{Ideal polyhedral surfaces}
\label{ssc:ideal-dual}

For Delaunay circle patterns, or equivalently ideal polyhedral surfaces in hyperbolic ends, the proofs have to be somewhat modified. For edge lengths, or induced metrics, the proof needs to be adapted because of several differences with, say, compact polyhedral surfaces. For dihedral angles, a different type of argument needs to be used, since (apparently) the scheme used for the other cases does not work directly.

Some of the main differences with compact or hyperideal polyhedral surfaces are as follow. 

\begin{itemize}
\item All edges have infinite length. However, if one chooses a horosphere $H_v$ centered at each vertex $v$ of $\Sigma$, then the length of an edge can be defined as the (oriented) length of the segment between the intersection with the horospheres centered at its two endpoints.
\item Those edge length are then well-defined up to the addition of a constant to all edges adjacent to a given vertex, because any two horospheres centered at the same vertex are at constant distance.
\item Dually, the angles between the faces of the dual surface $\Sigma^*$ are not well-defined, because those faces are isotropic in $dS^3$. However, one can choose for each face $v^*$ of $\Sigma^*$ a horosphere $H^*_v$ centered at $v$ in $dS^3$. Since two horospheres make up a constant angle, the angle between the horosphere corresponding to two advance faces of $\Sigma^*$ is well-defined.
\item Those horospheres in $dS^3$ are dual to the horospheres in $\HH^3$ centered at $v$. The angle between two horospheres $H^*_v$ and $H^*_{v'}$ is equal to the length of the segment between $H_v$ and $H_{v'}$, the horospheres in $\HH^3$ dual to  $H^*_v$ and $H^*_{v'}$.
\item Therefore, the angle between horospheres in $dS^3$ are well-defined up to changing the choice of horosphere at $v$, which adds a constant to all the angles of $H^*_v$.
\end{itemize}

We first describe the proof of Theorem \ref{tm:ideal-m}, on ideal polyhedral surfaces having a prescribed induced metric. 

\begin{enumerate}
\item As above, Theorem \ref{tm:ideal-m} on the induced metrics of ideal polyhedral surfaces can be reduced to Theorem \ref{tm:ideal*} on the edge lengths of those surfaces.
\item A first-order deformation of an horosphere $H_v$, centered at an ideal point $v$, is uniquely determined by an affine function $u$ on $H_v$. Given two horospheres $H_v$ and $H'_{v'}$ centered at  ideal points $v$ and $v'$, and a geodesic $e$ with endpoints $v$ and $v'$, the first-order variation of the length of the segment of $e$ between $H_v$ and $H'_{v'}$ is the sum of the evaluations of $u$ and $u'$ at intersections of $e$ with $H_v$ and $H'_{v'}$. 
\item Let $\Sigma$ be an ideal polyhedral surface in a hyperbolic end $\bE$, and let $\Sigma'$ be a first-order deformation of $\Sigma$ preserving the edge lengths of $\Sigma$ in $\R^E/i(\R^V)$. One can then choose first-order deformations of the horosphere foliations associated to the vertices of $\Sigma$ so that the first-order variation of the edge lengths, defined in $\R^E$ (rather than in $\R^E/i(\R^V)$) are zero.
\item The condition that the edge lengths are preserved means that the affine functions associated to the horospheres at the two vertices sum to zero. This provides an orientation to each edge (unless the functions are zero at both edges). The argument by Pak mentioned above can then be applied, to show that there is no non-trivial deformation of $\Sigma$, within a fixed edge $\bE$, which fixes all edge lengths, as an element of $\R^E/i(\R^V)$.  
\item For each vertex $v$ of $\Sigma$, the horospheres at $v$ can be identified through the normal flow. We denote by $\cH_v$ the vector space associated with any of those horospheres (with $\cH$ also the vector space associated to the other horospheres, through the normal flow). We denote by $\cH^*$ the dual vector space, of parallel 1-forms on any of the horospheres centered at $v$. 
\item A first-order deformation of $\Sigma$ within a fixed hyperbolic end $E$ is an element of the kernel of a linear operator (denote by $\phi$ in Section \ref{ssc:defo-ideal}), which assigns to an element of $\oplus_{v\in V}\cH^*_v$ (a 1-form for each vertex) an element of $\R^E/i(\R^V)$ (a weight assignment to the edges, up to the choice of an additive constant for each vertex.
\item First-order isometric deformations of $\Sigma$ in a {\em variable} hyperbolic end can be parameterized by the first-order variations of the dihedral angles, subject at each vertex $v$ to two linear conditions. Those linear conditions define an operator, which is the adjoint of the linear operator describing the first-order deformations in a fixed hyperbolic end. Its co-image is therefore zero, and this implies that the dimension of its kernel is always $6g-6$. It follows that $\Pidmet$ is a manifold of dimension $6g-6$. The fact that its projection to $\cE_S$ is an immersion then follows from projective rigidity, that is, rigidity within a fixed hyperbolic end. 
\end{enumerate}

For ideal polyhedral surfaces with prescribed dihedral angles, the proof of projective rigidity -- rigidity within a fixed hyperbolic end -- follows the same line, based on the argument of Pak \cite{pak:short} mentioned above, as for compact polyhedral surfaces, applied to the dual surfaces in $\cE^*$.

However, the proof of Theorem \ref{tm:ideal} that we use follows a different route -- using the same  argument as for compact polyhedral surfaces might be possible, but some care is needed since the dihedral angles of the dual surface are well-defined only up to the addition of a constant for each face. However there are two (other) possible ways to show that $\Pidang$ is a manifold of dimension $6g-6$.
\begin{itemize}
\item Once projective rigidity is knowns, it follows from a recent result of Wayne Lam \cite[Cor. 1.4]{lam:pullback}.
\item One can use the similar results for ideal polyhedral surfaces with prescribed induced metric, together with the complex structure on the space of ideal polyhedral surfaces (see Section \ref{ssc:complex-ideal}) to obtain the result.
\end{itemize}

\section{Background material}
\label{sc:background}

We provide in this section background information that is useful throughout the paper. Most of the material here is standard. The point of view chosen in Section \ref{ssc:ends} on hyperbolic ends perhaps adds a little novelty.

\subsection{Complex projective structures on surfaces}
\label{ssc:cp1}

A complex projective structure on $S$ is a geometric structure with charts in $\CP^1$ and transition maps in $PSL(2,\C)$. We denote by $\cC_S$ the space of complex projective structures on $S$. We refer to \cite{dumas-survey} for a broad presentation of complex projective structures on surfaces. We will use, among other, the following properties.
\begin{enumerate}
\item Given a $\CP^1$-structure on $S$, it has a developing map $\rm{dev}_\sigma:\tilde S\to \CP^1$, well-defined up to left composition by an element of $\PSL(2, \C)$, and a holonomy representation $\rho:\pi_S\to PSL(2,\C)$, such that, for all $x\in \tilde S$ and all $\gamma\in \pi_1S$,
  $$ \rm{dev}_\sigma(\gamma.x)=\rho(\gamma)\rm{dev}_\sigma(x)~. $$
\item $\cC_S$ is a complex manifold of dimension $6g-6$, where $g$ is the genus of $S$. It is homeomorphic to a ball.
\item Each complex projective structure $\sigma \in \cC_S$ has an underlying complex structure, and this defines a projection map $\pi:\cC_S\to \cT_S$.
\item Given $c\in \cT_S$, there is a unique Fuchsian complex projective structure with underlying complex structure $c$ -- we will denote it by $\sigma_F(c)$.  
\item The projection $\pi$ in facts identifies $\cC_S$ as the total space of an affine bundle over $\cT_S$, with fiber over $c$ modelled on the vector space of holomorphic quadratic differentials on $(S,c)$. Given $\sigma, \sigma'\in \pi^{-1}(c)$, for $c\in \cT_S$, the Schwarzian derivative $\cS(\sigma, \sigma')$ of the identity map from $(S,\sigma)$ to $(S,\sigma')$ is a holomorphic quadratic differential on $(S,c)$, and the map from $\sigma\to \cS(\sigma_F(c),\sigma)$ identifies $\pi^{-1}(c)$ with $T^*_c\cT_S$, the space of holomorphic quadratic differentials on $(S,c)$. 
\end{enumerate}

\subsection{The symplectic structure $\omega$ on $\cC_S$}
\label{ssc:symplectic}

The identification above between $\cC_S$ and $T^*\cT_S$ can be used to pull back on $\cC_S$ the cotangent symplectic structure on $T^*\cT_S$. We denote by $\omega$ this pull-back symplectic structure on $\cC_S$.

\subsection{Circle patterns on surfaces with complex projective structures}
\label{ssc:delaunay}

There is a notion of circles on $\CP^1$ -- circles are the ideal boundaries in $\partial_\infty \HH^3$ of totally geodesic planes in $\HH^3$, and $\CP^1$ can be identified with $\partial_\infty \HH^3$.

Elements of $PSL(2,\C)$ act by isometries on $\HH^3$. This action extends by continuity to an action on $\CP^1$, by {\em M\"obius transformations}. Since isometries of $\HH^3$ send totally geodesic planes to totally geodesic planes, M\"obius transformations of $\CP^1$ send circles to circles.

There is therefore a well-defined notion of circles on $S$ equipped with a $\CP^1$-structure. Moreover, there is a well-defined notion of disk. (Note however that not all circles bounds a disk in this context.)

\subsection{Circle packings as Delaunay circle patterns}
\label{ssc:special}

Circles packings can be considered as Delaunay circle patterns, thanks to a ``trick'' due to Thurston. Indeed, let $C$ be a circle packing on $(S,\sigma)$, with nerve $\Gamma$ -- that is, the vertices of $\Gamma$ correspond to the disks bounded by the circles in $C$, and two vertices are connected by an edge if the corresponding disks are tangent. Assume that $\Gamma$ is the 1-skeleton of a triangulation. Each interstice in $C$ -- that is, each connected component of the complement of the union of the closed disks -- is bounded by segments of three circles $C_i, C_j$ and $C_k$, which are pairwise tangent. There is therefore another another circle, say $C^*_{ijk}$, which is orthogonal to $C_i, C_j$ and $C_k$, and intersects them at their tangency points. This construction, when done for all interstices, defines another circle packing $C^*$, with nerve the graph $\Gamma^*$ dual to $\Gamma$.

Taking all circles of $C$ and of $C^*$ leads to a Delaunay circle pattern $\bar C$, with incidence graph $\bar\Gamma$, where $\bar\Gamma$ is a bipartite graph having one vertex for each face and each of $\Gamma$, and an edge between the vertex corresponding to a vertex $v$ and to a face $f$ of $\Gamma$ if and only if $v$ is a vertex of $f$. The intersection angle between the corresponding circles of $\bar C$ is then always $\pi/2$, by construction.

Conversely, given a Delaunay circle pattern $\bar C$ with bipartite incidence graph $\bar\Gamma$, and with all intersection angles $\pi/2$, the circles of $\bar C$ can be split into two circle packings $C$ and $C^*$, so that $\bar C$ is obtained from  $C$ (or from $C^*$) by the construction above.

\subsection{Complex projective structures and hyperbolic ends}
\label{ssc:complex-ends}

There is a one-to-one correspondence (presumably due to Thurston) between complex projective structures on $S$ and hyperbolic end structures on $S\times [0,\infty)$. This relation is easier to visualize for complex projective structures $\sigma$ for which the developing map $\dev_\sigma$ is injective. Denoting by $\Omega=\dev_\sigma(\tilde S)$, the holonomy representation $\rho:\CP^1\to \CP^1$, extends as an action of $\pi_1S$ on $\HH^3$ by isometries, and this action is properly discontinuous on
$$ \tilde \bE = \HH^3\setminus CH(\CP^1\setminus \Omega)~, $$
where $CH$ denotes the hyperbolic convex hull. The quotient of $\tilde \bE$ by $\rho(\pi_1S)$ is precisely the interior of the hyperbolic end $\bE$ associated to $\sigma$.

Conversely, given the hyperbolic end $\bE$, its lift in $\HH^3$ is a domain with boundary at infinity $\Omega$, and then $\partial_\infty \bE$ can be identified with $\Omega/\rho(\pi_1S)$. Since the action of $\rho$ is by elements of $PSL(2,\C)$, $\partial_\infty\bE$ is equipped with a $\CP^1$-structure $\sigma$.

When the developing map $\dev_\sigma$ is not injective, the construction of the hyperbolic end associated to $\sigma$ is similar, but slightly more abstract. It can be done by considering, for each maximal open disk in $\tilde S$ (equipped with the $\CP^1$-structure $\tilde\sigma$ lifted from $\sigma$) an open hyperbolic half-space, with identifications between two half-spaces when the corresponding disks intersect. This defines a non-complete hyperbolic manifold, with boundary at infinity $(\tilde S,\tilde\sigma)$, on which $\pi_1S$ acts naturally, with quotient the interior of the hyperbolic end associated to $\sigma$. (A detailed construction can be found in \cite{kulkarni-pinkall}.)

\subsection{Hyperbolic ends and their projective extension}
\label{ssc:ends}

We have seen in Section \ref{ssc:complex-ends} how hyperbolic ends homeomorphic to $S\times (0,\infty)$ are in one-to-one correspondence with complex projective structures on $S$.

There is however another, related correspondence: between complex projective structures on $S$ and globally hyperbolic maximal compact (GHMC) 3-dimensional de Sitter spacetimes, as described in \cite{scannell}.

This correspondence is, again, easier to describe for a $\CP^1$-structures $\sigma$ on $S$ such that $\dev_\sigma$ is injective, with image $\Omega\subset \CP^1$. It can be visualized using the projective model of $\HH^3$ inside the unit ball $B(1)\subset \R^3$, normalized so that $\Omega$ is contained in the open upper hemisphere $S^2_+$. One can then consider the domain $\tilde{\bE^*}$ in $\R^3\setminus B(1)$ which is the intersection of all half-spaces containing $\Omega$. This domain is naturally equipped with a de Sitter metric, because the exterior of $B(1)$ can be naturally identified with a projective model of one hemisphere of the de Sitter space $dS^3$, see e.g. \cite[Section 2]{scannell}. Moreover $\pi_1S$ acts properly discontinuously and isometrically on $\tilde{\bE^*}$, with quotient a GHMC de Sitter spacetime which we call $E^*$. 

In the general case where the developing map of $\sigma$ is not injective, the construction of the associated GHMC de Sitter spacetime can still be done, somewhat as for the hyperbolic end associated to $\sigma$, see \cite{scannell}.

The hyperbolic end and GHMC dS spacetime can be constructed together. Again, the construction is easier to visualize when $\dev_\sigma$ is injective, with image $\Omega\subset \CP^1$. With the notations used above, $\rho$ acts properly discontinuously and projectively on $\tilde\bE\cup \Omega\cup \tilde\bE^*$. Since the action of $\rho$ is projective in the projective models of $\HH^3$ and $dS^3$, $\bar\bE$ is equipped with a real projective structure. This space can also be equipped naturally with a complex-valued function, defined from the Hilbert metric of the sphere (suitably complexified), as in \cite{shu}. The hyperbolic and de Sitter metrics on the two sides of $\Omega$ can be recovered in this manner. The quotient of $\tilde\bE\cup \Omega\cup \tilde\bE^*$ by $\rho(\pi_S1)$ is then a composite space,
$$ \bar \bE = \bE\cup \partial_\infty E \cup \bE^*~. $$

There is a well-known duality between the hyperbolic and de Sitter spaces, see e.g. \cite{HR,shu,fillastre-seppi:spherical}. This duality associates to each oriented totally geodesic plane in $\HH^3$ a point in $dS^3$, and to each space-like totally geodesic plane in $dS^3$ a point in $\HH^3$. It can be defined using the models of both $\HH^3$ and $dS^3$ as quadrics in the 4-dimensional Minkowski space $\R^{3,1}$, or in purely projective terms.

Through this duality, points in $\bE^*$ are dual to the totally geodesic planes immersed in $\bE$, while points in $\bE$ are dual to maximal open domains in space-like totally geodesic planes immersed in $\bE^*$. This can be seen using the projective definition duality, in the universal cover of $\bar\bE$, identified as above to $\tilde\bE\cup \Omega\cup \tilde\bE^*$.

This duality gives rise to a duality between convex surfaces in $\bar \bE$, defined in the following manner.

\begin{definition}
  A surface $\Sigma\subset \bar \bE$ is {\em convex} if it bounds a convex domain in $\bar \bE$ containing $\partial_0\bE$. 
\end{definition}

The convexity of the domain is understood in terms of the real projective structure on $\bar\bE$.

Given a convex surface $\Sigma\subset \bE$, the dual surface $\Sigma^*$ is defined as the set of points dual to the support planes of $\Sigma$. One can show that $\Sigma^*$ is again a convex surface in $\bE$. Moreover, the dual surface of $\Sigma^*$ is then $\Sigma$, as for surfaces in $\HH^3$ and $dS^3$. (It follows from the same arguments, see e.g. \cite{HR,shu}.)

A key property of this duality is that if two totally geodesic planes in $\bE$, oriented towards $\partial_\infty \bE$, intersect with angle $\theta$ (measured relative to their orientation, that is, with angle zero if they are coplanar with the same orientation) then the dual points are at space-like distance $\theta$ in $\bE^*$. As a consequence, if $\Sigma$ is a convex polyhedra surface in $\bE$, and if $\Sigma^*$ is the dual surface, then the edge lengths of $\Sigma^*$ are equal to the exterior dihedral angles of the correspondinig edges of $\Sigma$, and conversely.

This duality takes different forms for different types of polyhedral surfaces. For instance:
\begin{itemize}
\item If $\Sigma$ is a (convex) compact polyhedral surface in $\bE$, $\Sigma^*$ is a space-like convex surface in $\bE^*$, and conversely.
\item If $\Sigma$ is an ideal polyhedral surface, the planes dual to each vertex $v\in \partial_\infty\bE$ is the isotropic plane through $v$ in $\bE^*$, while the dual of each edge is a space-like edge in $\bE^*$. So $\Sigma^*$ is a polyhedral surface with all faces isotropic, but all edges space-like.  
\item If $\Sigma$ is a hyperideal surface -- that is, its vertices are in $\bE^*$ but all its edges intersect $\bE$ -- then $\Sigma^*$ has all edges space-like, but all its faces are time-like in $\bE^*$, and contain a region which is in $\bE$.
\end{itemize}

\subsection{Delaunay circle patterns and ideal polyhedral surfaces}
\label{ssc:ideal}


As was already mentioned, there is a one-to-one relation between Delaunay circle patterns in a surface equipped with a complex projective structure, and ideal polyhedral surfaces in the corresponding hyperbolic end. The relations should be clear from the following lemma.

\begin{lemma} \label{lm:ideal}
  Let $\bE\in \cE_S$ be a hyperbolic end, and let $\Sigma\subset \bE$ be an ideal polyhedral surface in $\bE$. Let $\sigma\in \cC_S$ be the corresponding complex projective structure on $S=\partial_\infty \bE$. The set of disks defined as the images by the Gauss map of the planes extending the faces of $\Sigma$ constitute a Delaunay circle pattern in $(S,\sigma)$, and the dihedral angles between the disks are equal to the exterior dihedral angles between the faces of $\Sigma$.
\end{lemma}

The proof of this lemma is basically straightforward and follows from the definition, see \cite{delaunay}.

It is also useful to relate ideal polyhedral surfaces to equivariant ideal polyhedral immersions in $\HH^3$. Indeed, let $\Sigma\subset \bE$ be an ideal polyhedral surface in $\bE$, and let $u:S\to \Sigma$ be a homeomorphism. Then $u$ lifts to a homeomorphism $\tilde u: \tilde S\to \HH^3$, which is a polyhedral map, but is not necessary embedded -- only immersed. (It is embedded if the developing map of the $\CP^1$-structure associated to $\bE$ is injective.) Moreover, this immersion $\tilde u$ is equivariant with respect to the holonomy representation $\rho:\pi_1S\to PSL(2,\C)$ of $\cE$, in the sense that
$$ \forall x\in \tilde S, \forall \gamma\in \pi_1S, \tilde u(\gamma.x)=\rho(\gamma)\tilde u(x)~. $$
Conversely, given such an immersed equivariant (convex) ideal polyhedral immersion, it defines a hyperbolic end in which its image appears as an ideal hyperbolic surface.

As mentioned above, the dual of an ideal hyperbolic surface $\Sigma\subset \bE$ is a surface in $\bE^*\cup\partial_\infty\bE$, and intersecting $\partial_\infty\bE$ exactly at the vertices of $\Sigma$. The faces of $\Sigma^*$ are isotropic in $\bE^*$, each containing a unique vertex of $\Sigma$ (its dual vertex), and its edges are all space-like.

\subsection{Compact polyhedral surfaces and their duals}
\label{ssc:compact-dual}

Similar properties hold for compact polyhedral surfaces in a hyperbolic end $\bE$. Those polyhedral surfaces are again related to equivariant polyhedral immersions of $\tilde S$ in $\HH^3$ -- we do not repeat the description of the relation here, since it is identical to that of ideal polyhedral surfaces. 
We repeat here for the reader's convenience the definition of the induced and dual metrics on a compact polyhedral surface, as well as the spaces of compact polyhedral surfaces with prescribed induced and dual metric.

\begin{definition} \label{df:metrics-compact}
  Let $\Sigma$ be a compact polyhedral surface in a hyperbolic end $E$.
  \begin{enumerate}
  \item The {\em induced metric} on $\Sigma$ is the pull-back on $\Sigma$ of the hyperbolic metric on $\cE$ by the canonical inclusion. Given a hyperbolic metric $h$ with cone singularities of angles less than $2\pi$ on $S$, we denote by $\Pcomet$ the space of pairs $(\bE,\Sigma)$, where $\bE\in \cE_S$ and $\Sigma$ is a compact polyhedral surface in $\bE$ with induced metric isotopic to $h$.
  \item The {\em dual metric} on $\Sigma$ is the induced metric on $\Sigma^*$. It is a spherical metric with cone singularities, of angle larger than $2\pi$, at the vertices of $\Sigma^*$, which correspond to the faces of $\Sigma$. Given a spherical metric $h^*$ with cone singularities of angles larger than $2\pi$ on $S$, we denote by $\Pcodua$ the space of pairs $(\bE,\Sigma)$, where $\bE\in \cE_S$ and $\Sigma$ is a compact polyhedral surface in $\bE$ with dual metric isotopic to $h^*$.
  \end{enumerate}
\end{definition}


\subsection{Hyperideal polyhedral surfaces and their duals}
\label{ssc:hyperideal}

Let now $\Sigma\subset \bar\bE$ be a hyperideal polyhedral surface in an hyperbolic end $\bar\bE$. The intersection of $\Sigma$ with $\bE$ is a non-compact hyperbolic surface, with induced metric a complete hyperbolic metric on $\Sigma$ minus $n$ disks, where $n$ is the number of vertices of $\Sigma$. Each end of this induced metric has infinite area.

By definition, each vertex $v$ of $\Sigma$ is in $\bE^*$, and is dual to a totally geodesic plane $v^*\subset \bE$, which intersects orthogonally all edges of $\Sigma$ adjacent to $v$. It is therefore possible to {\em truncate} $\Sigma$, by replacing a neighborhood of each vertex $v$ by the disk in $v^*$ bounded by $v^*\cap \Sigma$. As a consequence, each edge $e$ of $\Sigma$ has a well-defined length, which is the length of $e$ between the planes dual to its vertices.

\subsection{The complex structure on the space of ideal polyhedral surfaces}
\label{ssc:complex-ideal}

A feature which is specific to ideal polyhedral surfaces -- and will be used below -- is that there are natural complex coordinates on the space $\Pidn(\Gamma)$ of ideal polyhedral surfaces with combinatorics given by $\Gamma$, when $\Gamma$ is the 1-skeleton of a triangulation of $S$. (Non-triangulated surfaces can be included if one chooses a triangulation of the faces of degree larger than $3$.) This complex structure can be understood in terms of the cross-ratio parameters associated to the edges, see \cite[Section 4.1]{thurston-notes} and \cite[Def 1.1]{lam2019quadratic}.

Let $\Sigma$ be such a triangulated ideal polyhedral surface, and let $e$ be an oriented edge, connecting vertices $v_i$ and $v_j$. Let $v_k$ and $v_l$ be the other vertices of the triangles on the left and right of $e$, respectively. The cross-ratio parameter associated to $e$ is then
$$ cr_{ij} = \frac{(v_k-v_i)(v_l-v_j)}{(v_k-v_j)(v_l-v_i}~. $$
where the vertices are identified to their coordinate in any affine chart in $\CP^1$. (The definition does not depend on the choice of the affine chart, because the cross-ratio is invariant under M\"obius transformations.)

Those cross-ratio parameters satisfy two polynomial conditions at each vertex. Let $v_0$ be a vertex, and let $v_1, v_2,\cdots, v_n$ be the adjacent vertices, in cyclic order. Then
$$ \prod_{i=1}^n cr_{0i}=1~, $$
$$ cr_{01} + cr_{01}cr_{02} + \cdots + cr_{01}cr_{02}\cdots cr_{0n}=0~. $$
Solutions of those two equations provide local coordinates on $\Pidn(\Gamma)$, that is, any small deformation of the $cr_{ij}$ satisfying the two equations at every vertex corresponds to an ideal polyhedral surface close to $\Sigma$.

Another important property of those cross-ratio parameters is that the holonomy along any closed curve can be recovered as a polynomial expression of the cross-ratio parameters. As a consequence, the subspace of parameters corresponding to a fixed hyperbolic end forms a (complex) subvariety of $\Pidn(\Gamma)$, equipped with the complex structure defined by the coordinates, see \cite{lam2019quadratic}.

The cross-ratio coordinates associated to an edge have an interpretation in terms of the geometric data on $\Sigma$: the argument of $cr_{ij}$ is the angle at the edge $e=[v_iv_j]$, while $\log|cr_{ij}|$ is equal to the shear along $e$, that is, the oriented distance along $e$ between the orthogonal projection on $e$ of the opposite vertices of the triangles on the left and on the right of $e$, see \cite[Section 4.1]{thurston-notes} or \cite[Section 6.1]{lam2019quadratic}.

\section{Projective rigidity}

In this section we present the arguments showing that polyhedral surfaces in hyperbolic ends (whether ideal, compact or hyperideal) are projectively rigid with respect to their induced metric, that is, they cannot be deformed within a fixed hyperbolic end without changing their induced metric. The same is true of compact, ideal or hyperideal polyhedra surfaces with respect to their dual metric.

The proof is based on that given by Bonsante and Wolf \cite{bonsante-wolf:projective} for circle packings, but stated in a polyhedral context where the relation to Pak's original argument is clearly visible.

\subsection{An (incorrect) sketch of the argument}

We first provide a brief sketch of the argument, which only works under a simplying hypothesis. This sketch works only for compact (or hyperideal surfaces), we will see below how it can be adapted to prove projective rigidity of ideal polyhedral surfaces when the induced metric is prescribed.

In each case, one needs to prove that a polyhedral surface (either the surface $\Sigma\subset \bE$ or, when considering the dual metric, the dual surface $\Sigma^*\subset \bar \bE$) cannot be deformed isometrically. The first step is to decompose this surface into triangles, by triangulating its non-triangular faces. 

An infinitesimal deformation, in a given hyperbolic end, would then appear as a set of vectors associated to each of the vertices (or the dual vertices when considering the dual metric). This ``discrete vector field'' should leave invariant the lengths of all edges of the triangulated surfaces. This means that for each edge, the orthogonal projections of the vectors at the two endpoints point in the same direction along the edge.

Assuming those orthogonal projections are non-zero, they define an orientation of the edges. It follows from its definition, and from the convexity of the polyhedral surface, that, for every vertex $v$, if the edges at $v$ are (in cyclic) order $e_1, \cdots, e_n=e_0$, then either they are all oriented towards or away from $v$, or there exist indices $i_0, i_1$ such that the orientation is towards $v$ for $e_{i_0+1}, \cdots, e_{i_1}$ and away from $v$ for $e_{i_1+1}, \cdots, e_{i_0}$. Pak's  argument \cite{pak:short}, used for polyhedral surfaces in Euclidean space and extended by Bonsante and Wolf \cite{bonsante-wolf:projective} to surfaces of high genus, then draws a contradiction by counting in two different ways the number $c$ of changes of orientation from one edge to the next at a vertex.
\begin{itemize}
\item $c\geq t$, where $t$ is the number of triangles in the surface, because there is no way to orient the three edges of a triangles in such a way that the orientations are compatible at each vertex,
\item $c\leq 2v$, where $v$ is the number of vertices, because of the condition given above on the orientation. 
\end{itemize}
Let $e$ be the number of edges in the triangulation. We see that
\begin{itemize}
\item $t-e+v=2-2g<0$ by the Euler relation, so $2t-2e+2v=4-4g$. But $3t=2e$, so that $2v-t=4-4g<0$.
\item $t\leq c\leq 2v$, so that $2v-t\geq 0$.
\end{itemize}
This is the required contradiction.

\subsection{A combinatorial lemma}

In practice, this argument needs to be applied with some care, because the deformation vectors might be orthogonal to some edges, and those edges are then not oriented. For this reason, a more careful analysis is needed.

We consider a closed, oriented surface $\Sigma$ of genus at least $2$, with a quasi-simplicial triangulation $\tau$. A {\em decoration} is an orientation of some of the edges  of $\tau$. 

Given such a partial orientation, we will count the number of {\em change of orientations} at the corners of $\tau$. Specifically, we will consider, for each corner of $\tau$ (that is, for each pair $(f,v)$, where $f$ is a face and $v$ is a vertex of $f$), that the decoration has:
\begin{itemize}
\item no change of sign, if the edges are either both non-oriented, or both oriented away from $v$, or both oriented towards $v$,
\item half a change of sign, if one edge is oriented and the other is not,
\item one change of sign, if one edge is oriented towards $v$ and the other away from $v$.
\end{itemize}

We will say that a decoration is {\em tight}, if there are at most $2$ changes of orientation at each vertex, and at most $1$ at vertices where there are at least $3$ non-oriented vertices, or two non-oriented consecutive vertices. 

\begin{lemma}\label{lm:pak}
  Let $\Sigma$ be a closed oriented surface of genus at least $2$. The only tight decoration of any quasi-simplicial triangulation of $\Sigma$ is the trivial one, that is, the decoration for which no edge is oriented.
\end{lemma}

\begin{proof}
  First remove from $\Sigma$ all triangles where no edge is oriented. This leaves a (possibly disconnected) triangulated surface with boundary, for which all boundary edges non-oriented.  Consider a connected component of this surface with boundary. It is a surface of genus $g$ with $b$ boundary components, with $\chi=2-2g-b<0$. Let $e_b$ be the number of boundary edges. 

  Note that in this surface, each triangle has at least one oriented edge (since triangles with no oriented edge has been removed). It follows that each triangle has at least one change of orientation, indeed:
  \begin{itemize}
  \item if all three edges are oriented, the orientation cannot be compatible at all corners, so there is at least one change of orientation,
  \item if exactly one edge is non-oriented, there is half a change of orientation at each of its endpoints,
  \item if exactly one edge is oriented, there is half a change of orientation at each of its endpoints.
  \end{itemize}

  We now call $c$ the total number of changes of orientations. Then $c\leq 2|V|-e_b$, since by the definition of a tight decoration the total inversion at each boundary vertex is at most $1$, and $e_b$ is also the number of boundary vertices. Moreover, $c\geq |F|$ since any triangle with at least one oriented edge has at a sum of changes of orientation of at least $1$. 

  $$ |V|-|E|+|F| = 2-2g-b~, $$
  while
  $$ 3|F| = 2|E|-e_b~, $$
  because each interior edge bounds $2$ faces, while each boundary edge bounds only one face. As a consequence,
  $$ 2|V|-(3|F|+e_b)+2|F| = 4-4g-2b~, $$
  so
  $$ 2|V|-e_b = |F| + (4-4g-2b) < |F|~. $$

  This is a contradiction since $c\geq |F|$ while $c\leq 2|V|-e_b$.  
\end{proof}


\subsection{Projective rigidity of polyhedral surfaces}
\label{ssc:prorig}

Rather than treating each case separately (compact, ideal or hyperideal polyhedral surfaces, induced metric or dual metric) we consider directly a more general notion that will include most of the cases that we intend to study.

\begin{definition}[(General) polyhedral surface]
  Let $\bE\in \cE_S$ be a hyperbolic end, and let $\bar \bE$ be its projective extension, as seen in Section \ref{ssc:ends}. An {\em admissible polyhedral surface} in $\bar \bE$ is a triangulated convex polyhedral surface with vertices either in $\bE$ or in $\bE^*$ (that is, no ideal vertex) and such that no edge is contained in $\bE^*$ and light-like.
\end{definition}

In this definition, we consider admissible polyhedral surfaces that are triangulated, in other terms, if a polyhedral surface has non-triangular faces, then one needs to triangulate them by adding as edges some of the diagonals. It is however acceptable that a pair of adjacent triangular faces are coplanar.

For such a polyhedral surface, we can define the {\em length} of an edge $e$ as follows (see \cite{shu} for more details). We choose a lift of $e$ to an edge $\bar e$ in the universal cover of $\bar \bE$, which itself projects to a edge (still denote by $\bar e$) in the extension of $\HH^3$ by the de Sitter space. Let $\bar v_0$ and $\bar v_1$ be the two vertices of $\bar e$. Those two points can be considered as points in the Minkowski $\R^{3,1}$, of squared norm either $1$ (if in $dS^3$) or $-1$ (if in $\HH^3$). We then define (following e.g. \cite{shu}) the (complex) length of $e$ by the condition that
$$ \cosh(d(e)) = - \langle \bar v_0, \bar v_1\rangle_{3,1}~. $$
This definition clearly does not depend on the choice of the lift $\bar e$. In fact what will matter below is the fact that the deformations considered preserve the scalar product on the right-hand side.

\begin{theorem}[Projective rigidity of admissible polyhedral surfaces] \label{tm:proj-general}
  Let $\bE$ be a hyperbolic end, and let $\Sigma$ be an admissible, triangulated polyhedral surface in $\bar \bE$. Any first-order deformation of $\Sigma$ within $\bE$ that leaves invariant all edge lengths of $\Sigma$ is zero.
\end{theorem}

\begin{proof}
  Let $Z$ be an infinitesimal deformation of $\Sigma$ preserving the lengths of all edges. In other terms, $Z\in \oplus_{v\in V}T_v\bar \bE$, where $V$ denotes the set of vertices of $\Sigma$.

  Let $\bar e$ be a lift of an edge $e$ of $\Sigma$, as above, and let $Z_{\bar e_0}$ and $Z_{\bar e_1}$ be the lifts of $Z$ to the endpoints $\bar e_0$ and $\bar e_1$ of $\bar e$. Since $Z$ does not change the length of $e$ at first order, it does not change $\langle \bar e_0,\bar e_1\rangle_{3,1}$ at first order, so that
  $$ \langle \bar e_0, Z_{\bar e_1}\rangle_{3,1} +\langle Z_{\bar e_0}, \bar e_1\rangle_{3,1}=0~. $$
  Since $Z_{\bar e_0}$ and $Z_{\bar e_1}$ are tangent to $\bar E$, we have
  $$ \langle \bar e_0, Z_{\bar e_0}\rangle_{3,1} = \langle \bar e_1, Z_{\bar e_1}\rangle_{3,1}=0~, $$
  and therefore
  $$ \langle \bar e_0-\bar e_1, Z_{\bar e_1}\rangle_{3,1} +\langle Z_{\bar e_0}, \bar e_1 - \bar e_0\rangle_{3,1}=0~. $$
  Since $e$ is not light-like, it follows that the orthogonal projections of $Z_0$ and of $Z_1$ on $\bar e$ either are both equal to zero, or they have the same orientation. This defines (if the projection is non-zero) an orientation of $e$, which clearly does not depend on the choice of the lift $\bar e$ of $e$. We define this orientation to be between $e_0$ to $e_1$, if the scalar product of $Z_{e_0}$ with the unit vector at $e_0$ along $e$ in the direction of $e_1$ is positive.

  Let $v$ be a vertex of $\Sigma$, and let $e_1, \cdots, e_n$ be the edges originating from $v$, in cyclic order. Let $u_{v,e_i}$ be the unit vector at $v$ along $e_i$, and let $Z_v$ be the deformation vector at $v$. The orientation which is given to edge $e_i$ is away from $v$ if $\langle u_{v,e_i}, Z_v\rangle>0$, towards $v$ if $\langle u_{v,e_i}, Z_v\rangle<0$, and $e_i$ is non-oriented if $\langle u_{v, e_i}, Z_v\rangle=0$. Since the $u_{v,e_i}$ form the vertices of a convex polygon in the link of $v$, there exists a subset $\{ e_{i_0}, \cdots, e_{i_1}\}\subset \Z/n\Z$ (possibly empty) such that $\langle u_{v, e_i}, Z_v\rangle>0$ for $i\in \{ e_{i_0}, \cdots, e_{i_1}\}$, a subset $\{ e_{i_2}, \cdots, e_{i_3}\}\subset \Z/n\Z$ (possibly empty) such that $\langle u_{v, e_i}, Z_v\rangle<0$ for $i\in \{ e_{i_2}, \cdots, e_{i_3}\}$, with at most one index between $i_1$ and $i_2$ and at most one index between $i_3$ and $i_0$ such that $\langle u_{v, e_i}, Z_v\rangle=0$. So in total there are at most 2 changes of orientation around $v$.

  Similarly, if there are two consecutive vertices $e_1, e_2$ at $v$ which are non-oriented, then $Z$ is orthogonal to $u_{v,e_1}$ and to $u_{v,e_2}$. The convexity of the polyhedral surface at $v$ then implies that there is a subset $\{ i_1, \cdots, i_2\}\subset \Z/n\Z$ containing $1,2$ such that $Z$ is orthogonal to the $u_{v,e_i}$ for $i\in \{ i_1, \cdots, i_2\}$, and such that $\langle Z, u_{v,e_i}\rangle$ has constant sign (and is non-zero) for all $i\not\in \{ i_1, \cdots, i_2\}$. The same conclusion holds if there are three different edges orthogonal to $Z$ at $v$. In those two cases, there is at most one change of orientation around $v$.
  
  It follows that the decoration defined by the orientation of the edges is tight. We are now precisely in a position to apply Lemma \ref{lm:pak}, and to conclude that $Z=0$. 
\end{proof}

\medskip

The various statements on the (infinitesimal) projective rigidity of various types of surfaces, with respect to their induced metrics or dual metrics, follow.
\begin{itemize}
\item Dual metrics of ideal polyhedral surfaces: Theorem \ref{tm:proj-general} is applied to the dual polyhedral surface, which is a polyhedral surface in $\bE^*$, with all faces isotropic, but all edges space-like.
\item Induced metrics on compact polyhedral surfaces: all edges and vertices are contained in $\bE$.
\item Dual metrics on compact polyhedral surfaces: we consider the dual surface, all its vertices and edges are contained in $\bE^*$, and all edges are space-like.
\item Induced metrics on hyperideal polyhedral surfaces: all vertices are in $\bE^*$, all edges intersect $\bE$, and are therefore time-like in $\bE^*$.
\item Dual metrics on hyperideal polyhedral surfaces: we consider the dual surfaces, all its vertices and are in $\bE^*$, all edges are space-like.
\end{itemize}

For non-triangulated surfaces, the same argument can be used as long as one considers not only the lengths of the edges properly speaking, but also, in addition, the lengths of some of the diagonals of the non-triangulated faces, so that, after adding those diagonals, the surface is triangulated. This is particularly relevant when considering the induced metric on the polyhedral surface, rather than just the dihedral angles. In this case, one can choose arbitrarily the diagonals of the non-triangulated faces, since any non-trivial first-order variation of the induced metric must change, at first order, the length of at least one of the diagonals or edges. 


\subsection{Projective rigidity of ideal polyhedral surfaces with prescribed edge lengths}
\label{ssc:prorig-ideal}

Theorem \ref{tm:proj-general} does apply to the rigidity of ideal polyhedral surfaces when one prescribes that the edge lengths remain fixed, because in that case all vertices are ideal. We describe here how the argument in the proof of Theorem \ref{tm:proj-general} can be adapted to this case. The construction considered here will also be useful in Section \ref{ssc:defo-ideal} to show the manifold property for the space of ideal polyhedral surfaces with prescribed edge length or induced metric (Theorem \ref{tm:ideal*}).

Let $\Sigma$ be an ideal polyhedral surface in a hyperbolic end $\bE\in \cE_S$. For each vertex $v$ of $\Sigma$, there is a one-parameter family $(H_v(r))_{r\geq r_0}$ of horospheres ``centered'' at $v$, with $H_v(r)$ and $H_v(r')$ at distance $|r'-r|$ for all $r,r'\geq r_0$. Each of those horospheres, equipped with its induced metric, is isometric to the Euclidean plane. Moreover, we can consider the normal flow, which identifies $H_v(r)$ to $H_v(r')$ for all $r'\geq r\geq r_0$ by a homothety of factor $\exp(r-r')$.

\begin{definition} \label{lm:H}
  We denote by $H_v$ the vector space of 1-parameter families of vector fields $(V(r))_{r\geq r_0}$ such that
  \begin{itemize}
  \item for all $r\geq r_0$, $V(r)$ is a parallel vector field on $H_v(r)$,
  \item for all $r'\geq r\geq r_0$, $V(r')$ is the image of $V(r)$ under the normal flow.
  \end{itemize}
\end{definition}

In other terms, $H_v$ is a 2-dimensional vector space, and its elements are parallel vector fields on the horospheres, which are invariant under the normal flow.

\begin{definition} \label{lm:H*}
  We denote by $H^*_v$ the vector space of 1-parameter families of 1-forms $(\alpha(r))_{r\geq r_0}$ such that
  \begin{itemize}
  \item for all $r\geq r_0$, $\alpha(r)$ is a parallel 1-form on $H_v(r)$,
  \item for all $r'\geq r\geq r_0$, $\alpha(r')$ is the image of $\alpha(r)$ under the normal flow.
  \end{itemize}
\end{definition}

Clearly $H^*_v$ is dual to $H_v$, with a duality defined by
$$ \alpha(V)=\alpha(r)(V(r))~, $$
evaluated at any point of $H_v(r)$, and for any value of $r\geq r_0$.

Let now $\dot v$ be a first-order displacement of $v$, and let $\dot H_v(r)$ be a compatible first-order displacement of $H_v(r)$, for all $r\geq r_0$. (That is, $v$ remains the center of $H_v(r)$ at first order under those deformations.) 

Each $\dot H_v(r)$ can be considered as a normal vector field along $H_v(r)$, that is,
$$ \dot H_v(r) = f(r) N~, $$
where $N$ is the unit normal vector field to the $H_v(r)$, pointing towards $v$. Moreover, a first-order deformation of one of the horosphere, say $H_v(r)$, uniquely determines a first-order deformation of the equidistant horospheres, say $H_v(r')$ for $r'>r_0$, under the condition that, at first order, $H_v(r)$ and $H_v(r')$ remain at constant distance $|r'-r|$.

\begin{lemma}
  For all $r\geq r_0$, $f(r)$ is a linear function on $H_v(r)$. Moreover, it is invariant under the normal flow, that is, if $\phi_{r,r'}:H_v(r)\to H_v(r')$ is the diffeomorphism defined by following the normal flow, then $f(r)=f(r')\circ \phi_{r,r'}$.
\end{lemma}

\begin{proof}
  A direct computation can be made in the upper half-space model, normalizing so that $H_v(r)$ is for instance the horizontal plane of equation $z=1$. The first-order deformations of this horosphere is given precisely by linear functions. Moreover, the corresponding first-order deformations of the equidistant horospheres -- which in this model correspond to the other horizontal planes -- are obtained by taking the image of this linear function by the normal flow, which in this case just corresponds to the projection along the vertical directions.
\end{proof}

Notice that only the differential of $f(r)$ is determined by $\dot v$, while adding a constant term to $f(r)$ is equivalent to changing the labelling of the horosphere $H_v(r)$ centered at $v$ by translation of $r$.

We can now state a ``projective rigidity'' lemma, which is a step in the proof of Theorem \ref{tm:ideal*}.

\begin{lemma} \label{lm:prorig-ideal*}
  Let $\Sigma \subset\bE$ be an ideal polyhedral surface in a hyperbolic end $\bE$. Any first-order deformation of $\Sigma$ which leaves the edge lengths invariant at first order is zero.
\end{lemma}

As mentioned in the introduction, the edge lengths of $\Sigma$ here should be considered as an element of $\R^E/i(\R^V)$, since edge lengths are defined only after choosing a horosphere centered at each vertex.

\begin{proof}
  Assume that $Z$ is a first-order deformation of the vertices of $\Sigma$, leaving the edge lengths invariant as an element of $\R^E/i(\R^V)$. For each vertex $v$ of $\Sigma$ and each horosphere $H_v(r)$ centered at $v$, the vector $Z_v$ determines a parallel one-form $\omega_v(r)$ on $H_v(r)$ (invariant under the normal flow as $r$ varies). This 1-form is the differential of a linear function $f_v(r)$, also invariant under the normal flow as $r$ varies, which is well-defined up to the choice of an additive constant, which corresponds to a change of horosphere in the foliation $(H_r(v))_{r\geq r_0}$. By definition of $\R^E/i(\R^V)$, the constants corresponding to the various vertices can be chosen such that, for any choice of parameters $r_v, v\in V$, the distances between the horospheres $H_v(r_v)$ remain constant at first order in the deformation. This implies that if $v$ and $w$ are the endpoints of an edge $e$, then $f_v(r)(u_{v,e}(r_v)) + f_w(r)(u_{w,e}(r_w))=0$.
  
  For each vertex $v$, each edge $e$ adjacent to $v$, and each $r\geq r_v$, we denote by $u_{v,e}(r)$ the intersection of $e$ with $H_v(r)$. We then define an orientation of $e$ similarly as in Section \ref{ssc:prorig-ideal}: away from $v$ if $f_v(r_v)(u_{v,e}(r_v))>0$, towards $v$ if $f_v(r_v)(u_{v,e}(r_v))<0$, and non-oriented if $f_v(r_v)(u_{v,e}(r_v))=0$. This orientation does not depend on the choice of $r_v$, and the orientations at the vertices $v$ and $w$ of $e$ are compatible, since
  $$ f_v(r)(u_{v,e}(r_v)) + f_w(r)(u_{w,e}(r_w))=0 $$
because the distance between $H_v(r_v)$ and $H_w(r_w)$ remains constant at first order.

Let $v$ be a vertex of $\Sigma$, and let $e_1, \cdots, e_n$ be the edges originating from $v$. The convexity of $\Sigma$ at $v$ implies that $u_{v,e_1}, \cdots, u_{v, e_n}$ are the consecutive vertices of a convex polygon in $H_v(r_v)$. Since $f_v(r_v)$ is a linear function we have a situation similar to what occured in the proof of Theorem \ref{tm:proj-general}: there exists a subset $\{ e_{i_0}, \cdots e_{i_1}\}\subset \Z/n\Z$ (possibly empty) such that $f_v(r_v)(u_{v, e_i})>0$ for $i\in \{ e_{i_0}, \cdots e_{i_1}\}$, a subset $\{ e_{i_2}, \cdots e_{i_3}\}\subset \Z/n\Z$ (possibly empty) such that $f_v(r_v)(u_{v, e_i})<0$ for $i\in \{ e_{i_2}, \cdots e_{i_3}\}$, with at most one index between $i_1$ and $i_2$ and at most one index betwee, $i_3$ and $i_0$ such that $f_v(r_v)(u_{v, e_i})=0$. Therefore, there are at most two changes of orientation at $v$.

Moreover, the same argument as in the proof of Theorem \ref{tm:proj-general} shows that if there are at least three edges such that $f_v(r_v)(u_{v, e_i})=0$, or two consecutive edges with this property, then there is at most one change of orientation at $v$.

It then follows that the decoration defined by the orientation is tight. We can therefore apply Lemma \ref{lm:pak}, and obtain the proof of the lemma.
\end{proof}

\subsection{Projective rigidity of ideal polyhedral surfaces with prescribed induced metric}
\label{ssc:prorig-induced}

The induced metric on an ideal polyhedral surface equipped with an ideal triangulation is uniquely determined by the lengths of the edges, considered as an element of $\R^E/i(\R^V)$, see \cite{penner:decorated1}. It therefore follows from Lemma \ref{lm:prorig-ideal*} that, given an ideal polyhedral surface $\Sigma$ in a hyperbolic end $\bE$, there is no non-zero first-order deformation of $\Sigma$, within the fixed end $\bE$, that preserves at first order the induced metric on $\Sigma$.

\section{Manifold structure and variation of the angles}
\label{sc:manifold}

In this section we prove the results stating that the spaces of polyhedral surfaces in (varying) hyperbolic ends, with prescribed induced metric or dual metric, are manifolds of dimension $6g-6$. All cases are, again, handled in a rather uniform manner, except for one: the space of ideal hyperbolic surfaces with prescribed dihedral angles. 

In all other cases, the result follows from a description of the infinitesimal deformations of polyhedral surfaces of fixed edge lengths, in terms of the variations of their dihedral angles. We show that, at the infinitesimal level, this description is adjoint to the description of the deformations in a fixed hyperbolic end, again with fixed induced metric. Theorem \ref{tm:proj-general} therefore shows that the tangent space to the space of deformations in a varying hyperbolic end has the expected dimension.

Rather than providing a ``general'' result that would apply in all cases but lead to unnecessary technical considerations, we explain the argument for compact polyhedral surfaces of fixed induced metric in Section \ref{ssc:defo-compact}, and then notice that the argument extends to other situations.

\subsection{Compact polyhedral surfaces with fixed induced metric}
\label{ssc:defo-compact}

We now focus on the proof of the first part of Theorem \ref{tm:compact}, that is, proving that the space $\Pcomet$ of pairs $(\bE,\Sigma)$, where $\bE\in \cE_S$ is a hyperbolic end and $\Sigma$ is a compact polyhedral surface in $\bE$ with induced metric isotopic to $h$, is a manifold of dimension $6g-6$. We denote by $n$ the number of cone singularities of $h$, that is, the number of vertices of a polyhedral surface in $\Pcomet$.

We denote by $\Pcon$ the space of compact polyhedral surfaces with $n$ vertices. In other terms, $\Pcon$ is the space of hyperbolic ends $\bE$ equipped with $n$ points which are in convex position, that is, which are vertices of a common compact polyhedral surface in $\bE$. Then $\Pcomet\subset \Pcon$. 

It also follows from its definition that $\Pcon$ is a smooth manifold of dimension $12g-12+3n$. We will see below that $\Pcomet$ is a $C^1$ submanifold of $\Pcon$.

The proof uses the dihedral angles of the polyhedral surface considered, so it requires that those dihedral angles, and their first-order variation in a deformation, are well-defined. This in turns requires that the faces are non-isotropic, which explains why we do not use this argument for the proof of Theorem \ref{tm:ideal}.

Elements in $\Pcomet$ are in one-to-one correspondence with ``compact'' equivariant polyhedral isometric embeddings of $(S,h)$ into $\HH^3$, see Section \ref{ssc:compact-dual}. Those equivariant isometric embeddings are, in turns, determined by their combinatorics, induced metrics and dihedral angles. Therefore, first-order variations fixing the induced metric (but possibly varying the holonomy) are determined by the first-order variation of the dihedral angles, which however need to satisfy a condition at each vertex. This condition on the first-order variation of the dihedral angles is given by the following well-known statement (see e.g. \cite[Theorem $A_S$]{polygones}). We state this lemma for polygons in $S^2$, for vertices in $E$, a similar statements holds in the natural extension of $dS^2$ by two copies of $\HH^2$, denoted by $HS^2$ here, see \cite{colI}.

\begin{lemma} \label{lm:polygons}
  Let $p\subset S^2$ be a polygon, with vertices $u_1, \cdots, u_n$ and angles $\theta_1, \cdots, \theta_n$. In any first-order deformation of $p$ preserving all edge lengths (at first order), we have
  $$ \sum_{i=1}^n \dot\theta_i u_i=0~. $$
  Conversely, any first-order deformation $(\dot\theta_i)_{i=1, \cdots, n}$ of $(\theta_i)_{i=1, \cdots, n}$ satisfying this condition is realized by a first-order deformation of $p$.
\end{lemma}

We now consider the following operator.
$$
\begin{array}{rccc}
  \Psi: & \R^{E(\Sigma)} & \to & \oplus_{v\in V}T_v\bE \\
  & (\dot \theta_e)_{e\in E(\Sigma)} & \mapsto & \left(\sum_{v\in e}\dot \theta_eu_{v,e}\right)_{v\in V}~. 
\end{array}
$$
Here the sum on the right-hand side is over all edges $e$ containing $v$, and $u_{v,e}$ now describes, for $v$ a vertex of $e$, the unit vector at $v$ along $e$.

It follows from Lemma \ref{lm:polygons} that the elements of $\ker(\Psi)$ correspond precisely to the infinitesimal variations of the dihedral angles of $\Sigma$ corresponding to infinitesimal deformations of $\Sigma$, possibly varying the holonomy.

We also introduce another operator, related to the deformations within a fixed hyperbolic end as seen above. We now choose an orientation on each edge $e$ of $\Sigma$, and denote by $e_-$ and $e_+$ its endpoints (the choice of orientation is actually irrelevant, we just need notations for the endpoints).
$$
\begin{array}{rccc}
  \Phi: & \oplus_{v\in V}T_v\bE & \to &  \R^{E(\Sigma)} \\
  & (Z_v)_{v\in V}& \mapsto & (\langle u_{e_-,e}, Z(e_-)\rangle+\langle u_{e_+,e}, Z(e_+)\rangle)_{e\in E(\Sigma)}~. 
\end{array}
$$
We define two scalar products, one on $\oplus_{v\in V}T_v\bE$:
$$ \llangle (Z_v)_{v\in V},(Z'_v)_{v\in V}\rrangle = \sum_{v\in V} \langle Z_v, Z'_v\rangle~, $$
the other on $\R^{E(\Sigma)}$:
$$ \llangle (\dot \theta_e)_{e\in E(\Sigma)}, (\dot \theta'_e)_{e\in E(\Sigma)}\rrangle = \sum_{e\in E(\Sigma)} \dot \theta_e\dot \theta'_e~. $$

\begin{lemma}
  When those scalar products are used to identify $\oplus_{v\in V}T_v\bE$ and $\R^{E(\Sigma)}$ to their duals, $\Psi$ is the adjoint of $\Phi$.  
\end{lemma}

\begin{proof}
  Let $Z=(Z_v)_{v\in V}\in \oplus_{v\in V}T_v\bE$ and let $\dot\theta=(\dot \theta_e)_{e\in E(\Sigma)}\in \R^{E(\Sigma)}$. Then
  \begin{eqnarray*}
    \llangle \Phi(Z), \dot \theta\rrangle
    & = & \sum_{e\in E} (\langle u_{e_-,e}, Z(e_-)\rangle+\langle u_{e_+,e}, Z(e_+)\rangle)\dot \theta_e \\
    & = & \sum_{v\in V}\sum_{v\in e} \langle u_{v,e}, Z(v)\rangle \dot \theta_e \\
    & = & \sum_{v\in V}\sum_{v\in e} \langle  \dot \theta_e u_{v,e}, Z(v)\rangle \\
    & = & \llangle Z, \Psi(\dot \theta)\rrangle~.                                                
  \end{eqnarray*}
\end{proof}

\begin{proof}[Proof of Theorem \ref{tm:compact}]
  By definition, $\Pcomet$ can be considered as a subset of $\Pcon$. In the neighborhood of each point corresponding to a compact polyhedral surface $\Sigma$ which is triangulated, it is defined by $|E|$ equations (where $|E|$ is the number of edges of $\Sigma$), so we will able to conclude that it is a smooth manifold if we know that those $|E|$ equations are linearly independent.

  The dimension of $\Pcomet$ will then be equal to $\dim(\Pcon)-|E|=12g-12+3n-|E|$. If we again denote the set of (triangular) faces of $\Sigma$ by $F$, then $2|E| = 3|F|$. But $3|F|-3|E|+3=6-6g$ by the Euler relation, so that $|E|=6g-6+3n$, and the ``expected dimension'' of $\Pcon$ is therefore $6g-6$. To prove that $\Pcomet$ is a smooth manifold, we therefore only need to prove that its tangent space has dimension at most $6g-6$.

  At a point where $\Sigma$ is not triangulated, the same analysis can in fact be applied. Indeed, each non-triangular face can be triangulated in different ways, but an elementary argument shows that the intersection of the kernels of the derivative of the lengths of the different edges that can be added does not depend on the triangulation.
  
  We will use the identification of $T\Pcomet$ with $\ker(\Psi)$, seen above. It follows from the projective rigidity of compact polyhedral surfaces, seen in Theorem \ref{tm:proj-general}, that $\ker(\Phi)=0$. Since $\Phi$ and  $\Psi$ are adjoint, it follows that $\Psi$ is surjective.

  But $\Psi: \R^{E(\Sigma)} \to \oplus_{v\in V}T_v\bE$ has domain of dimension $|E|=6g-6+3n$ and target of dimension $3|V|$. Therefore $\dim \ker(\Psi)=6g-6$. We can now conclude that $\Pcomet$ is a $C^1$ manifold of dimension $6g-6$.

  Finally, the fact that the first factor projection from $\Pcomet$ to $\cE_S$ is an immersion follows from the manifold structure of $\Pcomet$, together with the projective rigidity seen in Theorem \ref{tm:proj-general}, applied in the special case of a compact polyhedral surface.

  The $C^{1,1}$ regularity of $\Pcomet$ follows from the $C^{1,1}$ regularity of the map that associates to a compact polyhedral surface its induced metric. This map is $C^{\infty}$ at all triangulated surfaces, but only $C^{1,1}$ at points where at least one face has at least 4 edges. This can be seen from Lemma \ref{lm:C11} below, which we state in a simple setting but extends to faces with more than four edges.
\end{proof}

Notice that we only claim here that $\Pcomet$ is $C^{1,1}$. This contrasts with for instance $\Pidang$, which is analytic. This $C^{1,1}$ regularity boils down to the following lemma.

\begin{lemma} \label{lm:C11}
  Let $\Sigma$ be a compact polyhedral surface, and let $f$ be a face of $\Sigma$ which is a polygon with vertices $v_1, v_2, v_3, v_4$. Under a displacement of $v_3$ along a line orthogonal to $f$, the distance (for the induced metric on $\Sigma$) between $v_1$ and $v_3$ is only $C^{1,1}$, while the other distances between the $v_i$ are $C^{\infty}$.
\end{lemma}

\begin{proof}
  Let $y$ be the coordinate of $v_3$ along the line orthogonal to $f$. The minimizing geodesic from $v_1$ to $v_3$ along $\Sigma$ is
  \begin{itemize}
  \item for $y\leq 0$, the union of two geodesic segments along two triangular edges (coplanar for $y=0$), of lengths respectively $x_0$ and $\rm{acosh}(\cosh(x_1)\cosh(y))$,
  \item for $y\geq 0$, a single geodesic segment, of length $\rm{acosh}(\cosh(x_0+x_1)\cosh(y))$.
  \end{itemize}
  It follows that for $y<0$,
  $$ d'_{13}(y) = \frac{\cosh(x_1)\sinh(y)}{2\sqrt{\cosh^2(x_1)\cosh^2(y)-1}}~, $$
  while for $y>0$,
  $$ d'_{13}(y) = \frac{\cosh(x_0+x_1)\sinh(y)}{2\sqrt{\cosh^2(x_0+x_1)\cosh^2(y)-1}}~. $$
  Clearly this derivative is Lipschitz but not $C^1$ at $y=0$, so the distance $d_{13}$ between $v_1$ and $v_3$ along $\Sigma$ is $C^{1,1}$ but not $C^2$. 
\end{proof}

\subsection{Compact polyhedral surfaces with fixed dual metric}
\label{ssc:defo-compact*}

The same argument can be used with very minor adaptations, but it needs to be applied to the dual surface $\Sigma^*\subset E^*$. A compact polyhedral surface $\Sigma$ has dual metric equal to $h^*$ if and only $h^*$ is the induced metric on the dual surface, and the arguments of the previous section can be used without change to this dual surface to prove Theorem \ref{tm:compact*}.

\subsection{Hyperideal polyhedral surfaces with fixed induced metric}
\label{ssc:defo-hyperi}

Again, the same argument can be used for hyperideal surfaces. We now consider a hyperideal polyhedral surface in the extension $\bar \bE$ of a hyperbolic end $\bE$. 

\subsection{Hyperideal polyhedral surfaces with fixed dual metric}
\label{ssc:defo-hyperi*}

Here again, the argument needs to be applied to the dual surface in $\bar \bE$. As seen in Section \ref{ssc:ends}, this dual surface has its vertices and edges in the de Sitter part $\bE^*$ of $\bar \bE$, while its faces intersect $\bE$. However the argument remains the same. 

\subsection{Ideal polyhedral surfaces with fixed edge lengths}
\label{ssc:defo-ideal}

For ideal polyhedral surfaces, the argument used for the proof of Theorem \ref{tm:compact} needs to be adapted. We consider such an ideal polyhedra surface $\Sigma$ in a hyperbolic end $\bE$.

For each vertex $v$ of $\Sigma$, the horospheres at $v$ can be identified through the normal flow. As earlier, we denote by $H_v$ the vector space associated with any of those horospheres (with $H_v$ also the vector space associated to the other horospheres, through the normal flow), and by $H_v^*$ the dual vector space, of parallel 1-forms on any of the horospheres centered at $v$, again with identification between different horospheres through the normal flow. Then $H^*_v$ is dual to $H_v$, as seen in Section \ref{ssc:prorig-ideal}.

A first-order deformation of $\Sigma$ within a fixed hyperbolic end $\bE$ is determined by the choice for each vertex $v$ of an element of $w_v\in H^*_v$, because two affine functions on a horosphere have the same differential if and only if they differ by a constant, that is, by an infinitesimal change of horosphere centered at $v$. More precisely, a first-order deformation of the data given by $\Sigma$ together with the choice of a horosphere $H_v(r)$ for each vertex $v$ is given by an affine function $f_v(r)$ on $H_v(r)$, but two such affine function differ by a constant if and only if they correspond to the same displacement of $v$ (but a different variation of $H_v(r)$), and their differentials are then the same.

For each vertex $v$ and edge $e$ ending at $v$, we denote by $u_{v,e}(r)$ the intersection with $H_v(r)$. Then $f_v(r)(u_{v_e}(r))$ is independent of $r$ (since we have seen that the functions $f_v(r)$ are invariant under the normal flow) and we denote it by $f_v(u_{v,e})$. Different choices of the deformation of the horosphere at $v$ (for a fixed choice of the displacement of $v$) differ by a constant added to $f_v(u_{v,e})$ for all edges ending at a given vertex $v$.

The ``projective rigidity'' of $\Sigma$ in $\bE$ is therefore equivalent to the existence of a non-zero vector in the kernel of
  $$
  \begin{array}{cccc}
       \phi: & \oplus_{v\in V}H_v^* & \to & \R^E/i(\R^V) \\
       & (w_v)_{v\in V} & \mapsto & [(f_v(u_{e_-, e})+f_v(u_{e_+,e}))_{e\in E}]~,
  \end{array}
  $$
where $f_v$ is any affine function on $H_v$ with $df_v=w_v$, and where the target space is the quotient $\R^E/i(\R^V)$ because elements of $i(\R^V)\subset \R^E$ correspond to different choices of $f_v$. Lemma \ref{lm:prorig-ideal*} then shows that $\ker(\phi)=0$.

First-order isometric deformations of $\Sigma$ in a {\em variable} hyperbolic end can be parameterized by the first-order variations of the dihedral angles, subject at each vertex $v$ to two conditions: the Gauss-Bonnet condition,
$$ \sum_{e\ni v} \dot \theta_e=0~, $$
and the condition coming from the fact that the deformation of the link preserves its edge lengths (up to homothety),
$$ \sum_{e\ni v} \dot \theta_e u_{v,e}=0~. $$
Therefore, those first-order isometric deformations are parameterized by the kernel of the linear operator
$$
\begin{array}{cccc}
  \psi: & (\R^E)_0 & \to & \oplus_{v\in V}H_v \\
        & (\dot \theta_e)_{e\in E} & \mapsto & (\sum_{e\ni v} \dot \theta_e u_{v,e})_{v\in V}
\end{array}
$$
where
$$ (\R^E)_0 = \{ (\dot\theta_e)_{e\in E}~|~ \forall v\in V, \sum_{e\ni v}\dot\theta_e=0 \}~. $$

However, $\psi$ is dual to $\phi$: if $(\dot \theta_e)_{e\in E}\in (\R^E)_0$ and $(w_v)_{v\in V}\in \oplus_{v\in V} H^*_v$, then
\begin{eqnarray*}
  \llangle\psi((\dot \theta_e)_{e\in E})), (w_v)_{v\in V}\rrangle & = & \sum_{v\in v} w_v(\sum_{e\ni v}\dot\theta_e u_{v,e}) \\
                                                                 & = & \sum_{e\in E} (w_{e_-}(u_{e_-,e})+w_{e_+}(u_{e_+,e})) \\
  & = & \llangle (\dot \theta_e)_{e\in E}), \phi((w_v)_{v\in V})\rrangle~. 
\end{eqnarray*}

Since $\phi$ has kernel reduced to $0$, $\psi$ has co-image reduced to $0$, and is therefore surjective. It follows that $\ker(\psi)$ has dimension $6g-6$ at all points, and that $\Pidlen$ is a manifold of dimension $6g-6$. The proof of Theorem \ref{tm:ideal*} follows from this, and from Lemma \ref{lm:prorig-ideal*}

\subsection{Ideal polyhedral surfaces with fixed dihedral angles}
\label{ssc:defo-ideal*}

We have seen in Section \ref{ssc:ideal} that ideal polyhedral surfaces with prescribed dihedral angles in hyperbolic ends are essentially equivalent to Delaunay circle patterns in surfaces with complex projective structures. Using this dictionary, the second part of Theorem \ref{tm:ideal} follows from the projective rigidity of ideal polyhedral surfaces with respect to their dihedral angles (as seen in Section \ref{ssc:prorig}) together with \cite[Cor. 1.4]{lam:pullback}.

There is however another, rather direct argument, based on the corresponding result for ideal polyhedral surfaces with prescribed induced metric, together with the complex structure on the space $\Pidn$ of ideal polyhedral surfaces with $n$ vertices. Let $\Sigma\in \Pidang$, with induced metric $h$, and let $X\in T_{\Sigma}\Pidmet$ be an infinitesimal deformation of $\Sigma$ which leaves the induced metric fixed at first order. This means that $X$ does not change, at first order, the shear parameters of $\Sigma$ along its edges. The infinitesimal deformation $iX$ of $\Sigma$ then does not change at first order the dihedral angles of $\Sigma$, so $iX\in T_\Sigma\Pidang$. Conversely, if $Y\in T_\Sigma\Pidang$, then $iY\in T_\Sigma\Pidmet$. We can therefore conclude that the space of infinitesimal deformation of $\Sigma$ which do not change the dihedral angles at first order is everywhere a vector space of dimension $6g-6$, so that $\Pidang$ is a submanifold of $\Pidn$ of dimension $6g-6$. The proof of Theorem \ref{tm:ideal}, with the exception of the Lagrangian property, follows from this and from the projective rigidity of ideal polyhedral surfaces with prescribed dihedral angles, as seen in Section \ref{ssc:prorig}.

\section{Lagrangian properties}
\label{sc:lagrangian}

In this section, we prove that the various submanifolds considered above, defined as spaces of hyperbolic ends containing a polyhedral surface (whether compact, ideal or hyperideal) with given induced metric or dual metric, are Lagrangian for the symplectic form $\omega$ defined in Section \ref{ssc:symplectic}.

Those results are all proved using variants of an argument, already used in different forms in e.g. \cite{cp,loustau:complex,loustau:minimal,mazzoli2019-2}, based on the variational properties of variants of the renormalized volume of hyperbolic ends. We describe below the argument for compact polyhedral surfaces, and then only briefly describe how the argument can be adapted for ideal and hyperideal polyhedral surfaces.

\subsection{The (relative) renormalized volume of a hyperbolic end}


We consider a hyperbolic end $\bE\in \cE_S$, and a closed surface $\Sigma_0\subset \bE$, which bounds a geodesically convex subset in $\bE$ (this condition is not really necessary). To each metric $h$ in the conformal class at infinity, Epstein \cite{epstein:envelopes} associated a (possibly singular) surface $\Sigeps{h}\subset \bE$, defined as the envelope of the horospheres $H_v(h)$, which are the projections to $\bE$ of the sets of points $x\in \tilde \bE$ such that the visual metric from $x$ at $v$ is equal to the lft to $\partial_\infty\tilde{\bE}$ of $h$. If $h$ is replaced by $e^{2r}h$, for some $r>0$, the Epstein surface $\Sigeps{h}$ is replaced by $\Sigeps{e^{2r}h}$, which is at constant distance $r$ from $\Sigeps{h}$ towards $\partial_\infty\bE$.

Two key properties of those Epstein surfaces are as follows (for the second point, see e.g. \cite[Lemma 3.6]{compare}).
\begin{enumerate}
\item For any conformal metric $h$, there exists $r_0>0$ such that, for all $r\geq r_0$, $\Sigeps{e^{2r}h}$ is embedded, and between $\Sigma_0$ and $\partial_\infty\bE$.
\item If we define, for $r\geq r_0$,
  $$ W(e^{2r}h) = V(\Sigma_0, \Sigeps{e^{2r}h}) - \frac 14\int_{\Sigeps{e^{2r}h}}H da$$
  where $H=\tr(B)$ is the mean curvature, and $V(\Sigma_0, \Sigeps{e^{2r}h})$ is the volume of the region of $\bE$ between $\Sigma_0$ and $\Sigeps{e^{2r}h}$, then for all $r,r'\geq r_0$, 
  $$ W(e^{2r'}h) - W(e^{2r}h) = (r'-r)\pi |\chi(S)|~. $$
\end{enumerate}
This allow us to define, for any conformal metric $h$ on $\partial_\infty\bE$,
$$ W(h) = W(e^{2r}h) - r\pi  |\chi(S)|~, $$
for any $r\geq r_0$. In particular, we define
$$ V_R(\Sigma_0) = W(h_{-1})~, $$
where $h_{-1}$ is the hyperbolic conformal metric on $\partial_\infty\bE$. Heuristically, $V_R(\Sigma_0)$ is the renormalized volume of the region of $\bE$ situated between $\Sigma_0$ and $\partial_\infty\bE$. 

A key property of this renormalized volume is its first variation formula, which can be obtained using exactly the same arguments as in \cite[Section 6]{volume}.

\begin{lemma}
  Under a first-order deformation of $\bE$, including a first-order deformation of $\Sigma_0$ in $\bE$, 
  $$ \dot V_R(\Sigma_0) = - \mathrm{Re}(\langle \dot c,q\rangle) - \frac 12 \int_{\Sigma_0}\left(\dot H +\frac 12 \langle \dot I, \II\rangle_I\right)da_I~, $$
  where $q$ is the holomorphic quadratic differential on $\partial_\infty\bE$ defined as the Schwarzian of the uniformization map.
\end{lemma}

The first term on the right-hand side corresponds to the variation of the conformal metric at infinity, while the second term corresponds to the deformation of $\Sigma_0$.

\subsection{Compact polyhedral surfaces}

We can now use the renormalized volume defined in the previous section to define variants that will be used in the proof of the Lagrangian property in different settings.

\subsubsection{The image of $\Pcomet$ in $\cE_S$ is Lagrangian}
\label{sssc:comet}

Let $h$ be a hyperbolic metric with $n$ cone singularities, of angle less than $2\pi$, on $S$.

To prove that the image of $\Pcomet$ is Lagrangian in $\cE_S$, we define another, closely related version of the renormalized volume, depending as a function of $(\Sigma, \bE)\in \Pcon$, that is, $\bE$ is a hyperbolic end and $\Sigma$ is a compact polyhedral surface in $\bE$ with $n$ vertices.

One can then choose $\Sigma_0$ to be in the compact region of $\bE$ bounded by $\Sigma$, and define
\begin{equation}
  \label{eq:sum}
 \VRcomet(\Sigma,\bE) = V_R(\Sigma_0) + V(\Sigma_0,\Sigma) -\frac 12\sum_{e\in E}l_e\theta_e~, 
\end{equation}
where $V(\Sigma_0,\Sigma)$ is the volume of the ``slice'' of $\bE$ between $\Sigma_0$ and $\Sigma$. Clearly the function $\VRcomet(\Sigma,\bE)$ does not depend on the choice of $\Sigma_0$.

The variational formula in \cite{sem,sem-era}, together with the polyhedral limit argument there, shows that, under a first-order variation, the term corresponding to $\Sigma$ is that of the classical Schl\"afli formula for first-order deformation of compact hyperbolic polyhedra:
$$ \dot V(\Sigma_0,\Sigma) =  \frac 12 \int_{\Sigma_0}\left(\dot H +\frac 12 \langle \dot I, \II\rangle_I\right)da_I + \frac 12 \sum_{e\in E} l_e\dot \theta_e~. $$
It then follows by adding the variations of the terms on the right-hand side of \eqref{eq:sum} 
that $\VRcomet$ has a simple variational formula:
$$ \dot  \VRcomet(\Sigma,\bE) = - \mathrm{Re}(\langle \dot c,q\rangle) - \frac 12\sum_{e\in E} \dot l_e \theta_e~. $$
For first-order deformation which are tangent to the image of $\Pcomet$, $\dot l_e=0$ for all edges, and therefore, on the image of $\Pcomet$
$$ d\VRcomet = - \lambda~, $$
where $\lambda$ is the Liouville form of the symplectic form $\omega$ on $\cE_S$.

It follows that, still on the image of $\Pcomet$,
$$ \omega = d\lambda = -d(d\VRcomet) = 0~, $$
which shows that the image of $\Pcomet$ is Lagrangian in $(\cE_S,\omega)$.

\subsubsection{The image of $\Pcodua$ in $\cE_S$ is Lagrangian}
\label{sssc:codua}

Let $h^*$ be a spherical metric on $S$, with $n$ cone singularities of angle larger than $2\pi$ and closed, contractible geodesics of length larger than $2\pi$. 

The argument used above in Section \ref{sssc:comet} can be adapted to prove that the image of $\Pcodua$ in $\cE_S$ is Lagrangian. It is based on a slightly simpler functional, namely
$$ \VRcodua(\Sigma,\bE) = V_R(\Sigma_0) + V(\Sigma_0,\Sigma)~. $$
The variational formula for this function is
$$ \dot  \VRcodua(\Sigma,\bE) = - \mathrm{Re}(\langle \dot c,q\rangle) + \frac 12\sum_{e\in E} l_e \dot \theta_e~. $$

For infinitesimal deformations tangent to the image of $\Pcodua$, the dihedral angles (which are the lengths of the edges of the dual surface) are constant, so again
$$ d\VRcodua(\Sigma,\bE) = - \lambda~. $$
It then follows, as above, that $\omega$ vanishes on the image of $\Pcodua$, which is therefore Lagrangian.

\subsection{Ideal polyhedral surfaces}

To prove the Lagrangian property in Theorem \ref{tm:ideal}, we introduce a very similar functional. Now $\Sigma$ is an ideal polyhedral surface, and we define again
$$ \VRidang(\Sigma,\bE) = V_R(\Sigma_0) - V(\Sigma_0,\Sigma)~. $$
The variational formula for this functional has a term on $\Sigma$ corresponding to the Schl\"afli formula for ideal polyhedra,
$$ \dot  \VRidang(\Sigma,\bE) = - \mathrm{Re}(\langle \dot c,q\rangle) + \frac 12\sum_{e\in E} l_e \dot \theta_e~. $$
Here the edge lengths $l_e, e\in E$ are defined by choosing a horosphere centered at each vertex, but the formula doesn't depend on the choice of horosphere, since the sum of the (exterior) dihedral angles at the edges coming out of any vertex is always equal to $2\pi$.

For variations tangent to the image of $\Pidang$, $\dot \theta_e=0$ for all $e\in E$. It follows again that, on the image of $\Pidang$,
$$ d\VRidang = -\lambda~, $$
so that the image of $\Pidang$ is Lagrangian in $(\cE_S,\omega)$. 

\subsection{Hyperideal polyhedral surfaces}

For hyperideal polyhedral surfaces, the argument is exactly the same as for compact polyhedral surfaces, but one needs to consider the truncated hyperideal surfaces, as seen in Section \ref{ssc:hyperideal}.

\bibliographystyle{alpha}
\bibliography{/home/jean-marc/Dropbox/papiers/outils/biblio}
\end{document}